\documentclass{amsart}
\usepackage[mathscr]{eucal}
\usepackage{graphicx}
\usepackage{amscd}	
\usepackage{amsmath}
\usepackage{amsthm}
\usepackage{amsxtra}
\usepackage{mathtools}
\usepackage{calc} 
\usepackage{amsfonts}
\usepackage{amssymb}
\usepackage{latexsym}
\usepackage{pdfsync}

\usepackage{tikz-cd}
\usepackage[all]{xy}
\usepackage[colorlinks, bookmarks=true]{hyperref}

\newtheorem{theorem}{Theorem}[subsection]
\theoremstyle{plain}

\newtheorem{corollary}[theorem]{Corollary}
\newtheorem{corollary-definition}[theorem]{Corollary-Definition}

\newtheorem{definition}[theorem]{Definition}

\newtheorem{lemma}[theorem]{Lemma}

\newtheorem{proposition}[theorem]{Proposition}

\numberwithin{equation}{section}

\newcommand{\blank}{\hspace{0.04cm} \rule{2.4mm}{.4pt} \hspace{0.04cm} }

\DeclareMathOperator{\ot}{\overset{\rightharpoondown}{\otimes}}

\theoremstyle{definition}
\newtheorem{example}[theorem]{Example}
\newtheorem{remark}[theorem]{Remark}

\newcommand{\Ker}{\mathrm{Ker}\,}
\newcommand{\Coker}{\mathrm{Coker}\,}

\newcommand{\Hom}{\mathrm{Hom}\,}
\DeclareMathOperator{\Ext}{\mathrm{Ext}}
\newcommand{\Tor}{\mathrm{Tor}}

\newcommand{\Z}{\mathbb{Z}\,}
\newcommand{\Q}{\mathbb{Q}\,}

\newcommand{\lra}{\longrightarrow}

\newcommand{\injt}{\mathfrak{s}}
\newcommand{\injc}{\mathfrak{q}}

\newcommand{\T}{\overset{\rightharpoondown}{\mathrm{T}}}
\newcommand{\V}{\mathrm{V}}
\newcommand{\M}{\mathrm{M}}

\renewcommand{\O}{\Omega}
\newcommand{\s}{\Sigma}
\newcommand{\te}{\infty}
\newcommand{\dotr}{\mbox{$\boldsymbol{\cdot}$}}

\mathchardef\mhyphen="2D

\newcommand{\ga}{\alpha}
\newcommand{\gb}{\beta}

\renewcommand{\ge}{\varepsilon}

\newcommand{\gk}{\kappa}
\newcommand{\gl}{\lambda}

\newcommand{\gt}{\tau}

\newcommand{\ab}{\mathrm{Ab}\,}

\newcommand{\modL}{\mathrm{mod}\mhyphen\Lambda}
\newcommand{\LMod}{\Lambda\mhyphen\mathrm{Mod}}

\newcommand{\fp}{{fp}(\LMod, \ab)}

\newcommand{\dia}[1]{\[\xymatrix{#1 }\]}
\usepackage{pgfpages}
\usepackage{tikz}
\usetikzlibrary{matrix,arrows,snakes,backgrounds}

\newcounter{hours}
\newcounter{minutes}

\begin{document}

\title[Injective stabilization of additive functors, III]{Injective stabilization of additive functors, III. Asymptotic stabilization of the tensor product}

\author{Alex Martsinkovsky}             
\address{Mathematics Department\\
Northeastern University\\
Boston, MA 02115, USA}
\email{a.martsinkovsky@northeastern.edu}
\author{Jeremy Russell}
\address{Department of Mathematics\\
Rowan University\\
201 Mullica Hill Rd, Glassboro, NJ 08028}
\email{russelljj@rowan.edu}
\thanks{The first author is supported in part by the Shota Rustaveli National Science Foundation of Georgia Grant NFR-18-10849.}
\date{\today, \setcounter{hours}{\time/60} \setcounter{minutes}{\time-\value{hours}*60} \thehours\,h\ \theminutes\,min}
\subjclass[2010]{Primary: 16 E30 ; Secondary: }
\keywords{additive functor, injective stabilization of the tensor product, asymptotic stabilization of the tensor product, Tate cohomology, Tate homology, Vogel cohomology, Vogel homology, Buchweitz cohomology, connected sequence of functors, Mislin's P-completion, Triulzi's J-completion, connecting homomorphism, comparison homomorphism, Steenrod-Sitnikov homology, asymptotic torsion and cotorsion.}

\begin{abstract}
The injective stabilization of the tensor product is subjected to an iterative procedure that utilizes its bifunctor property. The limit of this procedure, called the asymptotic stabilization of the tensor product, provides a homological counterpart of Buchweitz's asymptotic construction of stable cohomology. The resulting connected sequence of functors is isomorphic to Triulzi's $J$-completion of the Tor functor. A comparison map from Vogel homology to the asymptotic stabilization of the tensor product is constructed and shown to be  always epic. The category of finitely presented functors is shown to be complete and cocomplete. As a consequence, the inert injective stabilization of the tensor product with fixed variable a finitely generated module over an artin algebra is shown to be finitely presented. A description of its defect and all right-derived functors is given. New notions of asymptotic torsion and cotorsion are introduced and are related to each other. A connection with Buchweitz cohomology based on injectives is established.
\end{abstract}

\maketitle
\tableofcontents

\section{Introduction}

This is the third in a series of papers on applications of what should be called ``homological algebra in degree zero'', this time to stable homological algebra. The specific application dealt with in this part is a new generalization of Tate homology (as opposed to Tate cohomology) to arbitrary modules over arbitrary rings. As an unexpected byproduct of this construct, we introduce new concepts of asymptotic torsion and asymptotic cotorsion, thus establishing a connection with the torsion and cotorsion introduced in the previous paper \cite{MR-2} of this series. 

Our interest in torsion stems from its importance for the problem of recognizing syzygy modules, as was expounded by M.~Auslander in~\cite{A67}. The significance of this problem transcends the boundary of algebra. In mathematical systems theory it is related to the controllability of linear systems. In topology, a similar problem is referred to as "delooping". Of special interest are infinite syzygy modules. Their topological counterparts are known as infinite loop spaces~\cite{Ad78}. Over Gorenstein local rings these are precisely maximal Cohen-Macaulay modules. In an obvious sense these are "asymptotic" objects and they are closely related to Tate cohomology as was shown 
in~\cite{Bu}.

Even more surprisingly, the new results (and their proofs) show the advantages of focusing on the category of finitely presented (aka coherent) functors, due to its completeness and cocompleteness. This indicates the emergence of a "coherent homological algebra",  which is based on redefining classical homological constructions that involve colimits. Further applications of the new techniques will be given in a subsequent paper.

The original motivation for this papers (in fact, for the entire series) came from an obvious misbalance between generalized Tate cohomology and generalized Tate homology: the former admits (at least) three constructions whereas the latter -- only two. More precisely, the available homological constructions are obvious analogs of their cohomological counterparts, leaving Buchweitz's  generalization of Tate cohomology  without a homological analog (see the next section for more details). 

The main technical tool used in this paper is the notion of the injective stabilization of an additive functor from modules to abelian groups. This concept goes back to  foundational works of M. Auslander in the 1960s (\cite{A66} and \cite{AB}). Recall that the injective stabilization of an additive functor $F$ is defined as
the kernel of the natural transformation from $F$ to its zeroth right-derived functor. It is usually denoted by~$\overline{F}$. For a given module $B$,  $\overline{F}(B)$ can be easily computed: if $\iota : B \to I$ is monic with $I$ injective, then $\overline{F}(B) \simeq \Ker F(\iota)$. (See~\cite{MR-1} for a detailed treatment of this construct.) In particular, one can take $F : = A \otimes \blank$. Unlike the tensor product itself, its injective stabilization is not balanced. For this reason, the overline is replaced by a harpoon and $\overline{F}(B)$ becomes $A \ot B$, with the harpoon pointing to the active variable, i.e., the variable being embedded in an injective. The resulting expression is in fact a bifunctor. Let $\Omega$ denote the syzygy operation in a projective resolution and $\Sigma$ denote the cosyzygy operation in an injective resolution. While neither $\Omega$ nor $\Sigma$ is well-defined (their values on a module depend on the chosen resolutions), $\Omega^{i}A \ot \Sigma^{i}B$ is well-defined and is again a bifunctor. Moreover, there is a sequence of canonically defined natural transformations 
\[
\xymatrix
	{
	\ldots \ar[r] 
	& \Omega^2 A\ \ot\ \Sigma^2 B  \ar[r]^{\Delta_2} 
	& \Omega A\ \ot\ \Sigma B \ar[r]^{\Delta_1} 
	& A\ \ot\ B,
	}
\]
and the limit $\T_{0}(A,B)$ of this sequence, called here the \texttt{asymptotic stabilization of the tensor product}, yields the desired generalization of Tate homology. The above construct works for arbitrary modules over arbitrary rings. 

Now we give a brief outline of the paper. Section~\ref{S:stable cohomology} describes various generalizations of Tate cohomology. Section~\ref{S:stable homology} does the same for Tate homology. Section~\ref{S:lemmas} deals with three lemmas, all related to connecting homomorphisms in planar or spatial diagrams. In Section~\ref{S:first-const} we give the first construction of stable homology, as described above via the asymptotic stabilization of the tensor product. In the same section we show that the result is a connected sequence of functors with an explicitly constructed connecting homomorphism. Another construction of stable homology, based on the sequence $\Tor_{1}(\Omega^{i}A, \Sigma^{i+1}B)$, is given in Section~\ref{S:sec-const}. The result is again a connected sequence of functors, isomorphic to $\T_{0}(\blank, \blank)$. Its connecting homomorphism comes from a doubly infinite exact sequence obtained by splicing the familiar long sequence of the Tor functors in positive degrees and a long sequence of iterated injectives stabilizations of the tensor product. Yet another construction of stable homology is given in Section~\ref{S:third-const}. This time it is based on the sequence $S^{1}\Tor_{1}(\Omega^{i}A, \blank)$, where~$S^{1}$ denotes the right satellite. The resulting connected sequence of functors is also isomorphic to $\T_{0}(A, \blank)$ and has the additional advantage of being obviously isomorphic to the so-called $J$-completion of Tor, introduced and studied by M.~Triulzi~\cite{Tr}.
In Section~\ref{S:comparison}, we construct a comparison transformation from Vogel homology to $\T_{0}(\blank, \blank)$ and show that it is always epic. This evokes a similar situation in topology where one has an epic natural transformation from Steenrod-Sitnikov homology to \v{C}ech homology. Based on that analogy, the first author conjectured that the kernel of the comparison transformation from  Vogel homology to the asymptotic stabilization of the tensor product should be given by a suitable first derived limit. A recent result of I.~Emmanouil and P.~Manousaki~\cite{EM} shows that this is indeed the case. In Section~\ref{S:(co)comp} we show that the category of finitely presented functors is (co)complete. This sets the stage for 
Section~\ref{S:coherence}, where we establish the coherence of the inert asymptotic stabilization when the fixed argument satisfies certain finiteness conditions. All finitely generated modules over an artin algebra satisfy this condition. In the same section we compute the defect of the inert asymptotic stabilization, which allows us to determine all of its right-derived functors. These results are then applied to the torsion functor $\injt$ introduced in~\cite{MR-2}. This leads naturally to the introduction of an asymptotic torsion and asymptotic cotorsion, and we show that the defect of the 
asymptotic torsion is isomorphic to the asymptotic cotorsion of the ring viewed as a module on the opposite side. 

The surprising connection with topology mentioned above hints at further developments. Treating Buchweitz's cohomological construction as an analog of stable homotopy groups (both are byproducts of universal inversion of a functor) one is forced to look for a homotopy-theoretic counterpart of the asymptotic stabilization. The proper context for doing this is an axiomatic homotopy theory, or, even more broadly, quite general categories with suitable cylinders. The relevant results will appear elsewhere.
 
We follow the terminology and notation established in~\cite{MR-1}; the reader may benefit from reviewing that source. Some results contained in the present paper overlap with some results obtained by the second author in his PhD thesis~\cite{R-Thesis}.

\section{Stable cohomology}\label{S:stable cohomology} 
Tate cohomology was invented around 1950. We assume that the reader is familiar with this notion, but if they need to refresh their memory, we recommend~\cite{Br} and~\cite{CE}. In 1977, F.~T.~Farrell \cite{F} constructed a cohomology theory for groups of finite virtual cohomological dimension that, for finite groups, gave the same result as Tate cohomology. In the mid-1980s, R.-O.~Buchweitz \cite{Bu} constructed a generalization of Tate (and Farrell) cohomology that worked over arbitrary Gorenstein commutative rings.

\subsection{Vogel cohomology} 

At about the same time, Pierre Vogel \cite{G} came up with his own generalization of Tate cohomology, and while he was interested in arbitrary group rings, his approach actually worked over any ring. We now review that construction.

Let $\Lambda$ be a (unital) ring and $M$ and $N$ (left) $\Lambda$-modules. Choose projective resolutions $(\mathbf{P}, \partial) \longrightarrow M$ and $(\mathbf{Q}, \partial) \longrightarrow N$. Forgetting the differentials, we have $\Z$-diagrams $\mathbf{P}$ and $\mathbf{Q}$ of
left $\Lambda$-modules, together with a $\Z$-diagram $(\mathbf{P},\mathbf{Q})$ of abelian groups. The latter has $\prod_i \Hom(P_{i},Q_{i+n})$ as its degree $n$ component. It contains the subdiagram $(\mathbf{P},\mathbf{Q})_b$ of bounded maps,
whose degree $n$ component is $\coprod_i \Hom(P_{i},Q_{i+n})$. Passing to the quotient, we have a short exact sequence of diagrams
\[
0 \longrightarrow (\mathbf{P},\mathbf{Q})_b \longrightarrow (\mathbf{P},\mathbf{Q}) \longrightarrow (\widehat{\mathbf{P},\mathbf{Q}}) \longrightarrow 0
\]
The standard definition, $D(f) := \partial \circ f - (-1)^{\deg f} f \circ \partial$, yields  a differential on the middle diagram, which clearly restricts to a differential on the subdiagram of bounded maps. Thus the inclusion map is actually an inclusion of complexes, and the corresponding quotient becomes the quotient complex. By construction, the maps in this short exact sequence are chain maps between the constructed complexes. The $n$th Vogel cohomology group of $M$ with coefficients in $N$, where $n \in \Z$, is then defined as the $n$th cohomology 
group of the complex $(\widehat{\mathbf{P},\mathbf{Q}})$. We denote it by 
$\mathrm{V}^{n}(M,N)$.

\subsection{Buchweitz cohomology}

 As we mentioned before, Buchweitz was interested in a generalized Tate 
cohomology over Gorenstein rings, but his construction (actually, one of two proposed) turned out to work for any ring. We now describe his approach. Again, let $\Lambda$ be an arbitrary (unital) ring, $M$ and $N$ (left) $\Lambda$-modules, and $\Lambda$-Mod the category of left $\Lambda$-modules and homomorphisms. First, we pass to the category  $\Lambda$-\underline{Mod} of modules modulo projectives, which has the same objects as $\Lambda$-Mod, but whose morphisms $(\underline{M,N})$ are defined as the quotient groups $(M,N)/P(M,N)$, where $P(M,N)$ is the subgroup of all maps that can be factored though a projective module. The composition of classes of homomorphisms is defined as the class of the composition of representatives. One of the advantages of this new category is that the syzygy operation $\Omega$ on $\Lambda$-Mod becomes an additive endofunctor on $\Lambda$-\underline{Mod}. In particular, for $M$ and $N$ we have a sequence of homomorphisms of abelian groups
\[
(\underline{M,N}) \longrightarrow (\underline{\Omega M,\Omega N}) \longrightarrow 
(\underline{\Omega^2 M, \Omega^2 N}) \longrightarrow \ldots
\] 
The $n$th Buchweitz cohomology group $\mathrm{B}^{n}(M,N)$, $n \in \Z$ is defined as 
\[
\underset{n+k, k \geq 0}{\underrightarrow{\lim}}\,(\underline{\Omega^{n+k} M, \Omega^{k} N}).
\]

\subsection{Mislin's construction}

Yet another generalization of Tate cohomology was given by G.~Mislin \cite{M} in 1994. It is a special case of a considerably more general construct. For a cohomological (or, more generally, connected) sequence of functors $\{F^i\}$, $i \in \Z$ Mislin uses a sequence of natural transformations
\[
F^i \longrightarrow S_{1}(F^{i+1}) \longrightarrow S_{2}(F^{i+2}) \longrightarrow \ldots,
\]
where $S_{j}$ denotes the $j$th left satellite, and defines what he calls the $P$-completion of $\{F^i\}$ as 
\[
\underset{k \geq 0}{\varinjlim}\, S_{k}(F^{i+k}) =: \mathrm{M}^{i}F.
\]
Evaluating the colimit on the group cohomology (viewed as a cohomological functor of the coefficients), he gets a new cohomological (or connected if the original sequence is connected but not necessarily cohomological) sequence of functors. He then proves that, for groups of finite virtual cohomological dimension, the new cohomology is isomorphic to Farrell cohomology. Moreover, he also establishes, for arbitrary groups, an isomorphism between his construction and Buchweitz's cohomology (called in the paper the Benson-Carlson cohomology, after the two authors, who independently found Buchweitz's cohomology in 1992). We remark that Mislin's construction is completely general and applies, in particular, to the $\Ext$ functors over any ring.
  
\section{Stable homology}\label{S:stable homology}

At this point, one may ask if there are \texttt{homological} analogs of the various cohomology theories discussed above. The answer to this question is less clear. First, there was no ``Tate homology'' in Tate's original work: only the Hom functor was used with complete resolutions. However, at the same time when P.~Vogel constructed his cohomology, he also constructed a homology theory. We begin by reviewing his construction.

\subsection{Vogel homology}

Let $\Lambda$ be a ring, $M$ a left $\Lambda$-module and $N$ a right $\Lambda$-module. Choose a projective resolution $(\mathbf{P}, \partial) \longrightarrow M$ and an injective resolution $ N \longrightarrow (\mathbf{I}, \partial)$. Forgetting the differentials, we have $\Z$-diagrams $\mathbf{P}$ and $\mathbf{I}$ of left and, respectively, right 
$\Lambda$-modules, together with a $\Z$-diagram $\mathbf{P} \widehat{\otimes} \mathbf{I}$ of abelian groups. The latter has $\prod_i (P_{i} \otimes I^{i - n})$ as its degree $n$ component. It contains the subdiagram $\mathbf{P} \otimes \mathbf{I}$, whose degree $n$ component is $\coprod_i (P_{i} \otimes I^{i - n})$. Passing to the quotient, we have a short exact sequence of diagrams
\[
0 \longrightarrow \mathbf{P} \otimes \mathbf{I} \longrightarrow \mathbf{P} \widehat{\otimes} \mathbf{I} \longrightarrow \mathbf{P} \overset{\vee}{\otimes} \mathbf{I}
\longrightarrow 0
\]
The standard definition  

\begin{equation}\label{vogel-differential}
 \begin{aligned}
 D(a \otimes b) 
 & := \partial_{P}(a) \otimes b + (-1)^{\deg a} a \otimes \partial_{I}(b)\\
 & = \big(\partial_{P} \otimes 1 + (-1)^{\deg_{1}(\blank)}1\otimes \partial_{I} \big)
 (a \otimes b),
\end{aligned}
\end{equation}
where $a$ and $b$ are homogeneous elements of $\mathbf{P}$ and, respectively, $\mathbf{I}$, and $\deg_{1}(\blank)$ picks the degree of the first factor of a decomposable tensor, gives rise to a differential on $\mathbf{P} \otimes \mathbf{I}$. It is easy to check that it extends to a differential, denoted 
by~$D$ again, on $\mathbf{P} \widehat{\otimes} \mathbf{I}$. Indeed, if 
$s \in (\mathbf{P} \widehat{\otimes} \mathbf{I})_{n}$ is a degree~$n$ element, then $s = (s_{i})_{i \in \Z}$, where each $s_{i} \in P_i \otimes I^{i - n}$ is just a finite sum of decomposable tensors. For  each $k \in \Z$, define
\[
D : \prod_i (P_{i} \otimes I^{i - n}) \longrightarrow (P_{k} \otimes I^{k +1 - n}):
s \mapsto (\partial \otimes 1)(s_{k+1}) + (-1)^k (1 \otimes \partial)(s_{k})
\]
Now, we obtain the desired differential by the universal property of direct product.

As a consequence, the third term in the short exact sequence above becomes a complex, and Vogel homology is now defined by setting
\begin{equation}\label{vogel-homology}
 \mathrm{V}_n(M,N) := \mathrm{H}_{n+1}(\mathbf{P} \overset{\vee}{\otimes} \mathbf{I}).
\end{equation}


\subsection{The $J$-completion}\label{S:J-completion}

A homological analog of Mislin's cohomological $P$-completion, called the $J$-completion, was defined by M.~Triulzi in his PhD thesis~\cite{Tr}.
Like its cohomological prototype, it is defined on connected sequences of functors, but even if the original sequence is cohomological, the result doesn't seem to be cohomological\footnote{This is related to the fact that the inverse limit is not an exact functor.}; one can only claim that the resulting sequence is connected. For reference, we denote it by $\mathrm{M}{_i}F$.


We summarize the existing constructions in the following table:
\smallskip 

\begin{center}
\begin{tabular}{|c|c|}
\hline 
\textbf{Cohomology} & \textbf{Homology} \\
\hline \hline
$\mathrm{V}^{i}(M,N)$ & $\mathrm{V}_{i}(M,N)$ \\ 
\hline
$\mathrm{B}^{i}(M,N)$ & ?  \\ 
\hline
$\mathrm{M}^{i}F$ & $\mathrm{M}_{i}F$\\
\hline
\end{tabular}
\end{center}
\bigskip
One of the goals of this paper is to replace the question mark by a homological analog of Buchweitz's construction.

\section{Three lemmas on connecting homomorphisms}\label{S:lemmas}

In this section, we gather general observations about connecting homomorphisms.

\textbf{Blanket assumption.} Throughout this paper, whenever we deal with the connecting homomorphism in the snake lemma, we automatically assume that the homomorphism was constructed by pushing and pulling the elements along a staircase path, as in the traditional proof of the lemma.
\smallskip

\subsection{Planar diagrams}

Let
\[
0 \longrightarrow B {\longrightarrow} C {\longrightarrow} D \longrightarrow 0
\quad \text{and}  \quad 0 \longrightarrow F {\longrightarrow} P {\longrightarrow} A \longrightarrow 0
\]
be short exact sequences with $P$ projective. Tensoring them together, we have a commutative diagram with exact rows and columns
 
\begin{equation}\label{D:Tor-iso}
\begin{gathered}
 \xymatrix
	{
	&  
	& 
	& 0\ar[d]
	&
\\
	&
	& 
	& \Tor_1(A, D) \ar[d]
\\
	& \Omega A\otimes B \ar[d] \ar[r]
	& \Omega A\otimes C \ar[d] \ar[r]
	& F \otimes D \ar[d] \ar[r] 
	& 0
\\
	0\ar[r]
	& P \otimes B \ar[r] \ar[d] 
	& P \otimes C \ar[r] \ar[d] 
	& P \otimes D \ar[d] \ar[r]
	& 0
\\	
	\Tor_1(A, D) \ar[r]^{\alpha}
	& A \otimes B \ar[d] \ar[r]^{\beta}
	& A \otimes C \ar[r] \ar[d]
	& A \otimes D \ar[d] \ar[r]
	& 0
\\
	& 0 
	& 0
	& 0
	}
\end{gathered}
\end{equation}
where the bottom row and the rightmost column are fragments of the corresponding long homology exact sequences. In this diagram, we have two copies of $\Tor_{1}(A, D)$, one at the top of the rightmost vertical exact sequence, the other -- as a term in the long exact sequence in the bottom row. Straight from the definition of the injective stabilization, one can see that both map into $A \ot B$ by way of the corresponding connecting homomorphisms, and we wish to make a commutative triangle by constructing an isomorphism between the two copies of Tor. 

\begin{lemma}\label{L:gamma}
 Let $\delta : \Tor_1(A, D) \to A \otimes B$ be the connecting homomorphism given by the snake lemma in the above diagram. Then there is an isomorphism 
 \[
 \gamma : \Tor_1(A, D) \longrightarrow \Tor_1(A, D)
 \]
 such that $\alpha \gamma = \delta$.
\end{lemma}

\begin{proof}
 If $C$ is projective, then we are immediately done by the snake lemma. Moreover, the construction of the isomorphism is explicit -- it is given by the connecting homomorphism. In general, choose an epimorphism $Q \longrightarrow D$ with $Q$ projective and lift the identity map on $D$ to obtain a commutative diagram with exact rows
 \[
\xymatrix
	{
	0 \ar[r]
	& \Omega D \ar[r] \ar[d]
	& Q \ar[r] \ar[d] 
	& D \ar[r] \ar@{=}[d]
	& 0
\\
	0 \ar[r]
	& B \ar[r] 
	& C \ar[r]  
	& D \ar[r] 
	& 0
	}
\] 
Tensoring it with the short exact sequence 
$0 \longrightarrow F \overset{\alpha}{\longrightarrow} P \overset{\beta}{\longrightarrow} A \longrightarrow 0$, we have a spatial commutative diagram with bottom face
\[
\xymatrix
	{
	0 \ar[r]
	& \Tor_1(A, D) \ar[r] \ar@{=}[d]
	& A \otimes \Omega D \ar[r] \ar[d]
	& A \otimes Q \ar[r] \ar[d] 
	& A \otimes D \ar[r] \ar@{=}[d]
	& 0
\\
	\Tor_1(A, C) \ar[r]
	& \Tor_1(A, D) \ar[r]
	& A \otimes B \ar[r] 
	& A \otimes C \ar[r]  
	& A \otimes D \ar[r] 
	& 0
	}
\] 
Its front face is the diagram~\eqref{D:Tor-iso}, and as we just observed, the lemma is true for the back face. The desired result now follows from the naturality of the connecting homomorphism and a trivial diagram chase.
\end{proof}

Next, we shall establish two results, stated and proved in greater generality than is needed for this paper, showing how to compose connecting homomorphisms running in spatial diagrams. For an exact $3 \times 3$ square and a connected sequence of covariant functors, the two possible compositions of the connecting homomorphisms, as in~(\cite[Proposition 4.1]{CE}), always anticommute. In the setup we are about to describe, there are no connected sequences of functors; instead, we postulate more general properties of the requisite spatial diagrams. The result, however, remains true -- the two possible compositions of the connecting homomorphisms anticommute. The proof, while somewhat tedious, is done by diagram chase and is thus elementary.

We continue to assume that a connecting homomorphism in the snake lemma is constructed by pushing and pulling the elements along a staircase path, as in the traditional proof of the lemma.

\subsection{Spatial diagrams: front, bottom, and right-hand faces}

\begin{lemma}\label{L:down-horizontal}
Let
\[
\xymatrix@R25pt@C35pt
	{
	&
	& L_{1}' \ar[rr] \ar[d] \ar[lddd]
	&
	& M_{1}' \ar@{->>}[rr] \ar[d] \ar[lddd]
	&
	& N_{1}' \ar[d] \ar[lddd]
\\
	&
	& L' \ar[rr] |! {[urr];[ddr]}\hole \ar@{->>}[d] \ar[lddd] |! {[ddl];[ddr]}\hole
	&
	& M' \ar@{->>}[rr]|*+[o][F-]{17} |!{[urr];[ddr]}\hole \ar@{->>}[d] 
	\ar[lddd]|*+[o][F-]{16} |! {[ddl];[ddr]}\hole
	&
	& N' \ar@[blue]@{->>}[d]|*+[o][F-]{18} \ar@[blue]@{>->}[lddd]|*+[o][F-]{19}
\\
	&
	& L_{2}' \ar[rr] |! {[uurr];[dr]}\hole |!{[urr];[ddr]}\hole 
	\ar[lddd] |! {[dl];[dr]} \hole |! {[ddl];[ddr]} \hole  
	&
	& M_{2}' \ar@[magenta]@{->>}[rr]|*+[o][F-]{20} |! {[uurr];[dr]}\hole |!{[urr];		[ddr]} \hole 
	\ar@[magenta]@{>->}[lddd] |! {[dl];[dr]} \hole |*+[o][F-]{21} |! {[ddl];[ddr]} 		\hole
	&
	& *+[F]{N_{2}'} \ar[lddd]
\\
	& L_{1} \ar[rr]  \ar[d] \ar@{->>}[lddd]
	&
	& M_{1} \ar@{->>}[rr]|*+[o][F-]{4} \ar[d]|*+[o][F-]{5} \ar@{->>}[lddd]
	&
	& N_{1} \ar@[blue][d]|*+[o][F-]{8} \ar@[blue]@{->>}[lddd]|*+[o][F-]{3}
	&
\\
	& L \ar[rr] |*+[o][F-]{14} |! {[urr];[ddr]}\hole \ar@{->>}[d] 
	\ar@{->>}[lddd]|*+[o][F-]{12} |! {[ddl];[ddr]}\hole
	&
	& M \ar@{->>}[rr]|*+[o][F-]{7} |! {[urr];[ddr]}\hole \ar@{->>}[d] |*+[o][F-]{15}
	\ar@{->>}[lddd] |*+[o][F-]{6} |! {[ddl];[ddr]}\hole
	&
	& N \ar@{->>}[d] \ar@{->>}[lddd]
	&
\\
	& L_{2} \ar@[magenta][rr]|*+[o][F-]{13} |! {[uurr];[dr]}\hole |! {[urr];[ddr]}\hole 
	\ar@[magenta]@{->>}[lddd] |! {[dl];[dr]} \hole |*+[o][F-]{11} |! {[ddl];[ddr]}\hole
	&
	& M_{2} \ar@{->>}[rr] |! {[uurr];[dr]}\hole  |! {[urr];[ddr]}\hole 
	\ar@{->>}[lddd] |! {[dl];[dr]} \hole |! {[ddl];[ddr]}\hole
	&
	& N_{2} \ar@{->>}[lddd]
	&
\\
	L_{1}'' \ar[rr] \ar[d]
	&
	& M_{1}'' \ar@[red]@{->>}[rr]|*+[o][F-]{1} \ar@[red][d]|*+[o][F-]{2}
	&
	& *+[F]{N_{1}''} \ar[d]|*+[o][F-]{\alpha}
	&
	&
\\
	L'' \ar@[red]@{>->}[rr] |*+[o][F-]{9} \ar@[red]@{->>}[d]|*+[o][F-]{10}
	&
	& M'' \ar@{->>}[rr] \ar@{->>}[d]
	&
	& N'' \ar@{->>}[d]
	&
	&
\\
	L_{2}'' \ar[rr] |*+[o][F-]{\beta}
	&
	& M_{2}'' \ar@{->>}[rr] 
	&
	& N_{2}''
	&
	&
	}
\] 
 
 be a commutative 3D diagram subject to the following conditions:
 
\begin{enumerate}

\item any three-term sequence with arrows running in the same direction is exact (i.e., exact at the middle term);

 \item each arrow preceded by an arrow in the same direction is epic;
 
 \item the three middle three-term sequences $L''M''N''$, $M_{2}'M_{2}M_{2}''$, and $N'NN''$  on the front, the bottom and the right-hand faces of this cube are short-exact, i.e., each sequence is exact in the middle, the first map is monic, and the second map is epic.
 \smallskip

\end{enumerate}
Then the image of the connecting homomorphism $\Ker \alpha \longrightarrow L_{2}''$ (in the front face) is in $\Ker \beta$, and the composition of the connecting homomorphisms 
$\Ker \alpha \longrightarrow L_{2}''$ and $\Ker \beta \longrightarrow N_{2}'$ (in the bottom face) equals the negative of the connecting homomorphism $\Ker \alpha \longrightarrow N_{2}'$ (in the right-hand face).
\end{lemma}

\begin{proof}
First of all, because of the assumptions on the three short exact sequences, the connecting homomorphisms mentioned in the statement are indeed defined. The first assertion is now immediate. To prove the second assertion, pick an element $n_{1}'' \in \Ker \alpha \subset N_{1}''$. Then $m'' := 2 \circ 1^{-1}(n_{1}'') =  6 \circ 5  \circ 4^{-1} \circ 3^{-1}(n_{1}'')$. We also have $m \in M$ such that $7(m) = 8 \circ 3^{-1} (n_{1}'') =: n$. Let $\mu := (14) \circ (12)^{-1} \circ 9^{-1} \circ 6 (m)$.
Since $15(m) = 0$, we have $m_{2} : = 15(\mu - m) = (13) \circ (11)^{-1} \circ (10) \circ 9^{-1} (m'')$. Since $7(\mu) = 0$, we have $7(\mu - m) = - 8 \circ 3^{-1} (n_{1}'')$.
By construction, $6(\mu) = 6(m)$, and therefore $\mu - m = 16(m')$ for some 
$m' \in M'$. Now $(19) \circ (17)(m') = -n$. Therefore, 
$(18) \circ (17) (m')$ is negative the value of the connecting homomorphism 
$\Ker \alpha \longrightarrow N_{2}'$ on $n_{1}''$. Using the commutativity of the diagram again, we have the same value for $(20) \circ (21)^{-1} (m_{2})$, which is the value of the composition of the other two connecting homomorphisms on the same element.
 \end{proof}

\subsection{Spatial diagrams: top, back, and left-hand faces}

Now we look at the composition of connecting homomorphisms in the three remaining planes of the cube.

\begin{lemma}\label{L:horizontal-down}
 Let
 \[
\xymatrix@R25pt@C35pt
	{
	&
	& L_{1}' \ar[rr] \ar[d] \ar[lddd]
	&
	& M_{1}' \ar@[purple]@{->>}[rr] |*+[Fo]{2} \ar@[red][d] |*+[Fo]{6}
	\ar@[magenta]@{>->}[lddd] |*+[Fo]{5}
	&
	& N_{1}' \ar[d] |*+[o][F-]{\beta} \ar[lddd] 
\\
	&
	& L' \ar@[red]@{>->}[rr] |*+[Fo]{7} |! {[urr];[ddr]}\hole 
	\ar@[purple]@{->>}[d] |*+[Fo]{3}
	\ar@[blue]@{>->}[lddd] |*+[Fo]{9} |! {[ddl];[ddr]}\hole
	&
	& M' \ar@{->>}[rr] |*+[Fo]{11} |!{[urr];[ddr]}\hole \ar@{->>}[d] 
	\ar@{>->}[lddd] |*+[Fo]{12} |! {[ddl];[ddr]}\hole
	&
	& N' \ar@{->>}[d] \ar[lddd]
\\
	&
	& *+[F]{L_{2}'} \ar[rr] |! {[uurr];[dr]}\hole |!{[urr];[ddr]}\hole 
	\ar[lddd] |! {[dl];[dr]} \hole |! {[ddl];[ddr]} \hole  
	&
	& M_{2}' \ar@{->>}[rr] |! {[uurr];[dr]}\hole |!{[urr]; [ddr]} \hole 
	\ar[lddd] |! {[dl];[dr]} \hole |! {[ddl];[ddr]} \hole
	&
	& {N_{2}'} \ar[lddd]
\\
	& L_{1} \ar@[magenta][rr] |*+[Fo]{4} \ar@[blue][d] |*+[Fo]{8} 
	\ar@[purple]@{->>}[lddd] |*+[Fo]{1} 
	&
	& M_{1} \ar@{->>}[rr] \ar[d] |*+[Fo]{13} \ar@{->>}[lddd]
	&
	& N_{1} \ar[d] \ar@{->>}[lddd]
	&
\\
	& L \ar@{>->}[rr] |*+[Fo]{10} |! {[urr];[ddr]}\hole \ar@{->>}[d] 
	\ar@{->>}[lddd] |! {[ddl];[ddr]}\hole
	&
	& M \ar@{->>}[rr] |! {[urr];[ddr]}\hole \ar@{->>}[d] 
	\ar@{->>}[lddd]  |! {[ddl];[ddr]}\hole
	&
	& N \ar@{->>}[d] \ar@{->>}[lddd]
	&
\\
	& L_{2} \ar[rr] |! {[uurr];[dr]}\hole |! {[urr];[ddr]}\hole 
	\ar@{->>}[lddd] |! {[dl];[dr]} \hole |! {[ddl];[ddr]}\hole
	&
	& M_{2} \ar@{->>}[rr] |! {[uurr];[dr]}\hole  |! {[urr];[ddr]}\hole 
	\ar@{->>}[lddd] |! {[dl];[dr]} \hole |! {[ddl];[ddr]}\hole
	&
	& N_{2} \ar@{->>}[lddd]
	&
\\
	*+[F]{L_{1}''} \ar[rr] |*+[o][F-]{\alpha} \ar[d] |*+[o][F-]{\gamma}
	&
	& M_{1}'' \ar@{->>}[rr] \ar[d]
	&
	& N_{1}'' \ar[d]
	&
	&
\\
	L'' \ar[rr]  \ar@{->>}[d]
	&
	& M'' \ar@{->>}[rr] \ar@{->>}[d]
	&
	& N'' \ar@{->>}[d]
	&
	&
\\
	L_{2}'' \ar[rr]
	&
	& M_{2}'' \ar@{->>}[rr] 
	&
	& N_{2}''
	&
	&
	}
\] 
be a commutative 3D diagram subject to the following conditions:

\begin{enumerate}

\item any three-term sequence with arrows running in the same direction is exact;

 \item each arrow preceded by an arrow in the same direction is epic;
 
 \item the three middle three-term sequences $M_{1}'M_{1}M_{1}''$, $L'M'N'$, and $L'LL''$  on the top, the back,  and the left-hand  faces of this cube are short-exact, i.e., each sequence is exact in the middle, the first map is  monic, and the second map is epic. Moreover, the two horizontal sequences 
 $LMN$ and $M'MM''$ passing through the center of the cube are also short-exact.
 \smallskip

\end{enumerate}
Then the image of $\Ker \alpha \cap \Ker \gamma$ under the connecting homomorphism $\Ker \alpha \longrightarrow N_{1}'$ (in the top face) is in $\Ker \beta$, and on $\Ker \alpha \cap \Ker \gamma$  the composition of the connecting homomorphisms 
$\Ker \alpha \longrightarrow N_{1}'$ and 
$\Ker \beta \longrightarrow L_{2}'$ (in the back face) coincides
with the connecting homomorphism $\Ker \gamma \longrightarrow L_{2}'$ (in the left-hand face).
\end{lemma}

\begin{proof}
First, we show that the connecting homomorphism 
$\Ker \alpha \longrightarrow N_{1}'$ maps $\Ker \alpha \cap \Ker \gamma$ to $\Ker \beta$. Pick $l_{1}'' \in \Ker \alpha \cap \Ker \gamma \subset L_{1}''$ and let
$m' :=  6 \circ 5^{-1} \circ 4 \circ 1^{-1} (l_{1}'')$.
We need to show that $11(m') = 0$ or, equivalently, that $m'$ is in the image of 7. Let $ l := 8 \circ 1^{-1} (l_{1}'')$. By the commutativity of the diagram, $12(m')  = 10 (l)$.
As $l_{1}'' \in \Ker \gamma$, we have $l = 9(l')$ for some $l' \in L'$, and therefore 
$12 \circ 7(l') = 10(l) = 12(m')$. Since 12 is assumed to be a monomorphism, $m' = 7 (l')$, which is the desired claim.

Now we can prove the second claim. Because each of the morphisms 1, 2, and 3 belongs to two connecting homomorphisms, it suffices to show that
$7^{-1} \circ 6 \circ 5^{-1} \circ 4 = 9^{-1} \circ 8$. But the two morphisms become equal when we precompose them with the 
monomorphism $10 \circ 9$.
\end{proof}

\section{The asymptotic stabilization: the first construction}\label{S:first-const}

\subsection{The construction}

Our next goal is to introduce what we shall call the \texttt{asymptotic stabilization} of the tensor product, which is a limit of a sequence of maps between injective stabilizations of tensor products of iterated syzygy and cosyzygy modules. This can be done in three equivalent ways, the first one being dealt with in this section.
\smallskip

We begin by constructing a homomorphism 
$\Omega A\, \ot\, \Sigma B \longrightarrow A\, \ot\, B$
of abelian groups, where $A$ is a right $\Lambda$-module and $B$ is a left $\Lambda$-module. Choosing a projective resolution $P^{\dotr} \to A$ and an injective resolution $B \to I^{\dotr}$ and tensoring the short exact sequences 
\[
0 \longrightarrow \Omega A \longrightarrow P_{0} \longrightarrow A \longrightarrow 0 \quad \textrm{and} \quad  0 \longrightarrow B \longrightarrow I^{0} \longrightarrow \Sigma B \longrightarrow 0,
\]
we have a commutative diagram of solid arrows whose rows, columns, and diagonal are exact:
\begin{equation}\label{D:big}
\begin{gathered}
 \xymatrix@R24pt@C14pt
 	{
	& 0 \ar[dr] 
	& 
	& 0\ar[d]
	&
\\
	&
	&\Omega A\ \ot\ \Sigma B\ar[dr]\ar@{>..>}[r]
	& \Tor_1(A,\Sigma B) \ar[d]
\\
	& \Omega A\otimes B \ar[d] \ar[r]
	& \Omega A\otimes I^0 \ar[d] \ar[r]
	& \Omega A\otimes 	\Sigma B \ar[d] \ar[r] \ar[dr]
	& 0
\\
	0\ar[r]
	& P_0\otimes B \ar[r] \ar[d] 
	& P_0\otimes I^0 \ar@{->>}[r] \ar[d] 
	& P_0\otimes \Sigma B \ar[d] \ar@{>->}[dr]
	& \Omega A\otimes I^1 \ar[d]
\\
	 \Tor_1(A,\Sigma B) \ar[r]^{\alpha}
	& A\otimes B \ar[d] \ar[r]
	& A\otimes I^0 \ar@{->>}[r] \ar[d]
	& A\otimes \Sigma B \ar[d]
	& P_0\otimes I^1
\\
	& 0 
	& 0
	& 0
	} 
\end{gathered}
\end{equation}
%
As $P_{0} \otimes \blank$ is an exact functor, the bottom southeast map is monic. Since the composition of this map with $\Omega A\ \ot\ \Sigma B \lra  \Omega A\otimes \Sigma B \lra P_0\otimes \Sigma B$ is zero, by the universal property of kernels, we have the dotted map, making the top triangle commute. Notice that this map is monic. Restricting the connecting homomorphism in the snake lemma, 
we have
%
%
a map $\Omega A\ \ot\  \Sigma B \lra A \ot B$. Iteration of this process yields a sequence
\begin{equation}\label{Eq:directed}
\xymatrix
	{
	\ldots \ar[r] 
	& \Omega^2 A\ \ot\ \Sigma^2 B  \ar[r]^{\Delta_2} 
	& \Omega A\ \ot\ \Sigma B \ar[r]^{\Delta_1} 
	& A\ \ot\ B.
	}
\end{equation}

Next we want to show that any two choices for $\Delta_{1}$, and thus for any other 
$\Delta_{i}$, are isomorphic. Let ${P^{\dotr}}' \to A$ and $B \to I{^{\dotr}}'$ some other resolutions.
Lifting the identity map on $A$, extending the identity map on $B$, and taking the tensor product results in a commutative 3D version of the diagram~(\ref{D:big}). By the naturality of the connecting homomorphism in the snake lemma, we have a commutative diagram with exact rows
\[
\xymatrix
	{
	\ldots \ar[r]
	& \Tor_1(A,\Sigma' B) \ar[rr] \ar[dd] \ar@{->>}[dr]
	&
	& A \otimes B \ar[r] \ar@{=}[dd]
	& A \otimes {I^{0}}' \ar[r] \ar[dd]
	& \ldots
\\
	&
	& A \ot' B 
	\ar[dd]^>>>>>>>>>>>>>>>>>>{\cong}_>>>>>>>>>>>>>>>>>>{\alpha} \ar@{>->}[ur]
	&
	& 
	&
\\
	\ldots \ar[r]
	& \Tor_1(A,\Sigma B) \ar ' [r][rr] \ar@{->>}[dr]
	&
	& A \otimes B \ar[r]
	& A \otimes {I^{0}} \ar[r]
	& \ldots
\\
	&
	& A \ot B \ar@{>->}[ur]
	&
	& 
	&
	}
\] 
By~\cite[Lemma 4.1)]{MR-1},
$\alpha$ is a canonical isomorphism, in fact, an equality.\footnote{This assertion presupposes that we work, as was announced, in a module category. The reader who wishes to work in a general abelian category, should not use the word ``equality''.} On the other hand, the right-hand side of the 3D-version of~(\ref{D:big}) yields a commutative diagram of solid arrows
\[
\xymatrix@R15pt@C9pt
	{
	&
	&
	&
	& \Tor_1(A,\Sigma' B) \ar@{>->}[dd] \ar[dl]
	&
	&
\\
	&
	&
	& \Tor_1(A,\Sigma B) \ar@{>->} [dd]
	&
	&
	&
\\
	& 0 \ar[r] 
	& \Omega' A\, \ot\, \Sigma' B \ar@/^3pc/@{>.>}[rruu] \ar ' [r] [rr] 
	\ar[dl]^>>>>>>>>>>>>{\beta}_{\cong}
	&
	& \Omega' A \otimes \Sigma' B \ar[rr] \ar[dl] \ar ' [d][dd]
	&
	& \Omega' A \otimes {I^1}' \ar[dl] \ar[dd]
\\
	0 \ar[r] 
	& \Omega A\,  \ot\,  \Sigma B \ar@/^3pc/@{>.>}[rruu] \ar[rr]
	&
	& \Omega A \otimes \Sigma B \ar[rr] \ar[dd]
	&
	& \Omega A \otimes {I^1} \ar[dd]
	&
\\
	&  
	& 
	&
	& {P_{0}}' \otimes \Sigma' B \ar@{>->} ' [r] [rr] \ar[dl] 
	&
	& {P_{0}}' \otimes {I^1}' \ar[dl] 
\\ 
	& 
	&
	& P_{0} \otimes \Sigma B \ar@{>->}[rr]
	&
	& P_{0} \otimes {I^1} 
	&
	}
\] 
with exact rows and columns. The dotted arrows also come from the diagram~(\ref{D:big}) and make the triangles containing them commute. By~\cite[Lemmas 9.1 and 9.2]{MR-1},
$\beta$ is a canonical isomorphism. Using the fact that the map 
$\Tor_1(A,\Sigma B) \longrightarrow \Omega A \otimes \Sigma B$ is  monic, we have that the curved square also commutes. Splicing it with the left-hand square containing $\alpha$ from the preceding diagram, we have a commutative square
\[
\xymatrix
	{
	\Omega' A\, \ot \, \Sigma' B \ar[r]^{\Delta_{1}'} \ar[d]^{\beta}_{\cong}
	& A \, \ot' \, B \ar[d]^{\alpha}_{\cong}
\\
	\Omega A\,  \ot\,  \Sigma B \ar[r]^{\Delta_{1}}
	& A \, \ot \, B
	}
\] 
with the vertical maps being canonical isomorphisms. This proves 
\begin{proposition}\label{P:any-two-choices}
 Any two choices for $\Delta_{1}$, and hence for any $\Delta_{i}$, based on the diagram~\eqref{D:big} are canonically isomorphic. \qed
\end{proposition}

Arguments very similar to the ones just used yield

\begin{proposition}\label{P:functorial}
 The homomorphism $\Delta_{1} : \Omega A\,  \ot\,  \Sigma B
 \longrightarrow A \, \ot \, B$, and hence any~$\Delta_{i}$, is functorial in both $A$ and $B$. \qed
\end{proposition}

For any integer $n$ (including negative values), the process of constructing the sequence~(\ref{Eq:directed}) may be repeated with $\Omega^{k+n} A$ in place of $\Omega^{k}A$, yielding sequences 
\begin{equation}
M_n(A,B):=\{\Omega^{k+n} A\ \ot\ \Sigma^k B , \Delta_{k+1}\}_{k,k+n \geq 0} 
\end{equation}

\begin{definition}
The {asymptotic stabilization $\T_n(A,\blank)$ of the left tensor product in degree $n$ with coefficients in the right $\Lambda$-module $A$} is  
\begin{align}
 \begin{split}
 \T_n(A,\blank)(B) & := \T_n(A,B) \\
& := \underset{k,k+n\geq 0}{\varprojlim} \Omega^{k+n} A\ \ot\ \Sigma^k B=\varprojlim M_n(A,B)
\end{split}
\end{align}
\end{definition}

It is easy to see that each $\T_{n}$ is a bifunctor additive in each variable. We shall say that $\T_{n}(A, \blank)$ is the \texttt{active} asymptotic stabilization. It is injectively stable.  We shall say that $\T_{n} (\blank, B)$ is the \texttt{inert} asymptotic stabilization. It is projectively stable.

%

Next we observe that the $\T_{n}$ allow infinite dimension shifts in both directions if one utilizes both arguments.

\begin{lemma}\label{L:dim-shift}
For all integers $n$ and all \texttt{nonnegative} integers $k$, 
there are canonical isomorphisms of functors
\[ 
	\T_n(A,\Sigma^k\blank) \cong  \T_{n-k}(A,\blank)
\]
and
\[
	\T_n(\Omega^{k}A,\blank) \cong  \T_{n+k}(A,\blank)
\]	

\end{lemma}

\begin{proof}
The sequences (including the structure maps) for the components of the former (respectively, latter) pair of functors at any right $\Lambda$-module can be obviously chosen to be shifts of each other. 
\end{proof}

Now we want to discuss the vanishing of the functors 
$\T_{\bullet}(A, \blank)$. The first result is an an immediate consequence of the definitions.

\begin{proposition}
 If the right global dimension of $\Lambda$ is finite then 
 $\T_{n}(A, \blank) = 0$ for all integers $n$. \qed
\end{proposition}

\begin{proposition}\label{P:flat-dim}
 If the flat dimension of $A$ is finite, then $\T_{n}(A, \blank) = 0$ for all integers $n$.
\end{proposition}

\begin{proof}
 As the diagram~\eqref{D:big} shows, we have an injection
$\Omega A\,  \ot\,  \Sigma B \longrightarrow 
 \Tor_{1}(A, \Sigma B)$. In particular, 
 $\Omega^{n+k} A\, \ot\, \Sigma^{k} B$, $n+k, k \geq 1$ embeds in  $\Tor_{1}(\Omega^{n+ k-1}A, \Sigma^{k} B)$. But the latter vanishes for $n + k - 1 \geq \textrm{fl.} \dim\, A$.
\end{proof}

It is known that the vanishing of stable \texttt{cohomology} in one degree implies its vanishing in all degrees. We do not know if a similar statement is true for $\T_{\bullet}(A, \blank)$. A partial result is provided by 

\begin{proposition}
 If $\T_n(A,\blank) = 0$ for some integer $n$, then $\T_m(A,\blank) = 0$ for all $m < n$.  If, in addition, $\Lambda$ is quasi-Frobenius, then $\T_m(A,\blank) = 0$ for all  $m \in \Z$. 
\end{proposition}

\begin{proof}
The first assertion is an immediate consequence of the first isomorphism of Lemma~\ref{L:dim-shift}. Suppose now that 
$\Lambda$ is quasi-Frobenius. Since projective modules are injective, for any positive integer $k$,  any right $\Lambda$-module $B$ is a $k$th cosyzygy module in an injective resolution of $\Omega^{k}B$, i.e., 
$B \simeq \Sigma^{k}\Omega^{k}B$. Therefore,
\[
\T_{n+k}(A,B) \cong \T_{n+k}(A, \Sigma^{k}\Omega^{k}B)
\cong \T_{n}(A, \Omega^{k} B) = 0.
\]
\end{proof}

\subsection{The connectedness property}

Now we want to define, for each short exact sequence 
$0 \longrightarrow B' \longrightarrow B \longrightarrow B'' \longrightarrow 0$
of left $\Lambda$ modules, connecting homomorphisms 
\[
\omega_{n} : \T_n(A, B'') \longrightarrow \T_{n-1}(A, B')
\] 
and show that ($\T_\bullet(A,\blank), \omega_{\bullet}$) is a connected sequence of functors.\footnote{Any sequence of additive functors can be made connected by choosing the zero map as the connecting homomorphism. Our choice will be nonzero.} 

We continue to assume that the connecting homomorphism in the snake lemma is defined by pushing and pulling elements along a staircase pattern, as in the standard proof of the lemma.

By Lemma~\ref{L:dim-shift}, it suffices to define 
\[
\omega_{1} : \T_1(A, B'') \longrightarrow \T_{0}(A, B').
\]
To this end, we use the horse-shoe lemma and construct a commutative diagram 
 
\begin{equation}\label{horizontal1}
\begin{gathered}
 \xymatrix
	{
	& 0 \ar[d]
	& 0 \ar[d]
	& 0 \ar[d]
	&
\\
	0 \ar[r]
	& B' \ar[r] \ar[d]
	&  {I^{0}}' \ar[r] \ar[d]
	& \Sigma B'  \ar[r] \ar[d]^{\gamma}
	& 0
\\
	0 \ar[r]
	& B \ar[r] \ar[d]
	& I^{0} \ar[r] \ar[d]
	& \Sigma B \ar[r] \ar[d]
	& 0
\\
	0 \ar[r]
	& B'' \ar[r] \ar[d]
	& {I^{0}}''  \ar[r] \ar[d]
	& \Sigma B'' \ar[r] \ar[d]
	& 0
\\
	& 0 
	& 0 
	& 0 
	&
	}
\end{gathered}
\end{equation}
where the rows and columns are exact, and the middle column is a split-exact sequence of injective modules. We will also need the embedding $0 \longrightarrow  \Sigma B' \overset{\epsilon}{\longrightarrow} {I^{1}}'$ from the injective resolution of $B'$. Tensoring this diagram with $\Omega A$ (and flipping it about the diagonal) gives us another commutative diagram  of solid arrows 
 
\begin{equation}\label{connect11}
\begin{gathered}
\xymatrix@R15pt@C10pt
	{
	&
	&
	&
	& 0 \ar[d]
	&
\\
	& 
	&
	& 
	& \Omega A\, \ot\, B'' \ar[d]
	&
\\
	& \Omega A \otimes B' \ar[rr] \ar[d]
	&
	& \Omega A \otimes B \ar[r] \ar[d]
	& \Omega A \otimes B'' \ar[r] \ar[d]
	& 0
\\
	0 \ar[r]
	& \Omega A \otimes {I^{0}}' \ar[rr] \ar[dd]
	&
	& \Omega A \otimes I^{0} \ar[r] \ar[dd]
	& \Omega A \otimes {I^{0}}'' \ar[r] \ar[dd]
	& 0
\\
	\Omega A\, \ot\, \Sigma B' \ar@{>->}[dr]
	&
	&
	&
	&
	&
\\
	& \Omega A \otimes \Sigma B' \ar[rr]^{1 \otimes \gamma} \ar[dd] 
	\ar[dr]^{1 \otimes \epsilon}
	&
	& \Omega A \otimes \Sigma B \ar[r] \ar[dd] \ar@{.>}[dl]
	& \Omega A \otimes \Sigma B'' \ar[r] \ar[dd]
	& 0
\\
	&
	& \Omega A \otimes {I^{1}}'
	&
	&
	&
\\
	& 0 
	&
	& 0 
	& 0 
	&
	}
\end{gathered}
\end{equation}
with exact rows, columns, and the lower-left diagonal. The snake lemma yields 
a map $\kappa :  \Omega A\, \ot\, B'' \longrightarrow  \Omega A \otimes \Sigma B'$
with $(1 \otimes \gamma) \circ \kappa = 0$. On the other hand, since ${I^{1}}'$ is injective, $\epsilon$ extends along $\gamma$ and therefore  $1 \otimes \epsilon$
extends along $1 \otimes \gamma$. Hence $(1 \otimes \epsilon) \circ \kappa = 0$,
and $\kappa$ factors through $\Ker (1 \otimes \epsilon) = \Omega A\, \ot\, \Sigma B'$. We have thus constructed a map $\kappa_{1}^{1} : \Omega A\, \ot\, B'' \longrightarrow 
\Omega A\, \ot\, \Sigma B'$. 
Similarly, we have maps $\kappa_{1}^{i} : \Omega^{i} A\, \ot\, \Sigma^{i-1} B'' \longrightarrow \Omega^{i} A\, \ot\, \Sigma^{i} B'$ for each natural $i$.

The next step, as one would expect, is to show that the maps $\kappa^{i}_{1}$ are compatible with the structure maps~$\Delta$. Actually, as we shall see, this is not true since the corresponding  squares anticommute rather than commute. This motivates 

\begin{definition}
 For each integer $i$, set $\omega^{i}_{1} := (-1)^{i}\kappa^{i}_{1}$.
\end{definition}

Notice that both the $\Delta$ and the $\kappa$ are connecting homomorphisms in suitable diagrams.


\begin{lemma}
 Under the above assumptions and notation, the diagram 

 \[
\xymatrix
	{
	\Omega^{2} A\, \ot\, \Sigma B'' \ar[r]^{\omega^{2}_{1}}
	\ar[d]^{\Delta_{1}}
	& \Omega^{2} A\, \ot\, \Sigma^{2} B' \ar[d]^{\Delta_{2}}
\\	
	\Omega A\, \ot\, B'' \ar[r]^{\omega^{1}_{1}} 
	& \Omega A\, \ot\, \Sigma B'
	}
\] 
commutes. 
\end{lemma}

\begin{proof}
Tensoring the commutative diagram~(\ref{horizontal1}) with the exact sequence 
\[
0 \longrightarrow \Omega^{2}A \longrightarrow P_{1} 
\longrightarrow \Omega^{1}A \longrightarrow 0,
\]
where $P_{1}$ is a projective module, we have a spatial commutative diagram satisfying all conditions of Lemma~\ref{L:down-horizontal}. In that diagram, the connecting homomorphism on the front face equals $\Delta_{1}$, and the connecting homomorphism on the bottom face equals 
$\kappa^{1}_{1} = - \omega^{1}_{1}$. By the lemma, the composition 
$\omega^{1}_{1} \Delta_{1}$ equals the connecting homomorphism on the right-hand vertical face. 

Now we shift all indices in the diagram~(\ref{horizontal1}) one step up and again tensor it  with the above short exact sequence. The resulting spatial diagram satisfies all conditions of 
Lemma~\ref{L:horizontal-down}. In that diagram, the connecting homomorphism on the top is $\kappa^{2}_{1} = \omega^{2}_{1}$, and the connecting homomorphism on the back equals 
$\Delta_{2}$. By the lemma, the composition $\Delta_{2} \omega^{2}_{1}$ equals the connecting homomorphism
on the left-hand vertical face. Since that face coincides with the right-hand vertical face of the former diagram, we have $\omega^{1}_{1} \Delta_{1} = \Delta_{2} \omega^{2}_{1}$.
\end{proof}

Applying the foregoing lemma repeatedly and passing to the limit, we have a homomorphism $\omega_{1} : \T_{1}(A, B'') \longrightarrow \T_{0}(A, B')$. 
As we observed before, the same construction yields $\omega_{n} : \T_{n}(A, B'') \longrightarrow \T_{n-1}(A, B')$ for all integers $n$.

\begin{theorem}\label{T:first-conn-sec}
The pair  $(\T_\bullet(A,\blank), \omega_{\bullet})$,  is a connected sequence of functors.
\end{theorem}

\begin{proof}
 We already remarked that $\T$ is an additive functor. Therefore, given an exact sequence of left $\Lambda$-modules $0 \longrightarrow B' \overset{\alpha}{\longrightarrow} B \overset{\beta}{\longrightarrow} B'' \longrightarrow 0$, 
the composition  
 \[
 \T_{n}(A,B') \overset{\T_{n}(A, \alpha)}{\longrightarrow} \T_{n}(A,B) \overset{\T_{n}(A,\beta)}{\longrightarrow}  \T_{n}(A,B'').
 \]
 of the induced maps is zero. The fact that 
 $\T_{n-1}(A,\alpha) \circ \omega_{n} = 0$ follows from the snake lemma applied to the diagram~\eqref{connect11}. For the same reason,  $\omega_{n} \circ \T_{n}(A,\beta) = 0$. Thus it remains to show that the $\omega_{n}$ are functorial. But this follows from the functoriality of the connecting homomorphism in the snake lemma applied to the diagram~\eqref{connect11}. 
\end{proof}

\section{The asymptotic stabilization: the second construction}\label{S:sec-const}

Next we want to show that the asymptotic stabilization $\T_{\bullet}(A,B)$ can  be computed via the $\Tor$ functors, completely bypassing the need to use the injective stabilization of the tensor product. This approach uses non-functorial tools but has the advantage of being formulated in terms of familiar operations. It will also offer an intuitive perspective on Proposition~\ref{P:flat-dim}. For that purpose, we are going to construct a commutative diagram 
\begin{equation*}
\xymatrix@C5pt
	{
	\ldots \ar[rr]
	&
	&\Tor_{1}(\Omega A, \Sigma^{2} B) \ar[rr]  \ar@{->>}[rd]
	&
	& \Tor_{1}(A, \Sigma B) \ar@{->>}[rd] \ar[rr]
	&
	& A \otimes B
\\
	& \ldots \ar[rr] \ar@{>->}[ru]
	&
	& \Omega A \ot \Sigma B \ar[rr] \ar@{>->}[ru]
	&
	& A \ot B \ar@{>->}[ru]
	}
\end{equation*}
where the bottom sequence is given by~\eqref{Eq:directed}, and the arrows in the top sequence are connecting homomorphisms. Clearly, once such a diagram has been constructed, the limits of the two horizontal sequences will be isomorphic, showing that the asymptotic stabilization can indeed be constructed using the $\Tor$ functors. Moreover, we shall also show that all northeast arrows are monic and all southeast arrows are epic. Thus all stages in the first construction of the asymptotic stabilization can be recovered via the epi-mono factorizations of the top arrows.

\subsection{The construction}

This construction requires explicit choices, so for a right $\Lambda$-module $A$ we choose a projective resolution $P_{\dotr} \lra A$ and use the definition of $\Tor_{1}(A, \blank)$ as the first left satellite of the tensor product, i.e., via the exact sequence 
\begin{equation}\label{Tor-def}
 0 \longrightarrow \Tor_{1}(A, \blank) \longrightarrow \Omega A \otimes \blank \longrightarrow P_{0} \otimes \blank \longrightarrow A \otimes \blank \longrightarrow 0.
\end{equation}
For a projective resolution of $\Omega A$ we choose the projective resolution of $A$ truncated in degree 1. This allows us to claim that 
$\Tor_{i+1}(A, \blank) = \Tor_{i}(\Omega A, \blank)$, where we do mean an equality rather than an abstract isomorphism. 

For a short exact sequence 
$0 \longrightarrow  C \longrightarrow D \longrightarrow E \to 0$ of left $A$-modules, recall the construction of the connecting homomorphisms $\Tor_{i+1}(A, E) \lra \Tor_{i}(A, C)$ in the corresponding long exact sequence of the $\Tor$ functors. The case $i = 0$ consists of evaluating the sequence~\eqref{Tor-def} on the short exact sequence above and then using the snake lemma. For positive values of $i$, we describe the construction when $i = 1$ and then use the dimension shift. To this end, we replace $A$ with $\Omega A$ and build a snake diagram as in the case $i = 0$. The new diagram and the original one have a common row,
$\Omega A \otimes C \longrightarrow \Omega A \otimes D \longrightarrow \Omega A \otimes E \longrightarrow 0$,
which allows to glue the two diagrams together:
\small
\begin{equation*}\label{glue}
\xymatrix@R20pt@C1pt
	{
	&
	&
	& \Tor_{1}(\Omega A, C) \ar[ldd] \ar[rr]
	&
	& \Tor_{1}(\Omega A, D) \ar[ldd] \ar[rr]
	&
	& \Tor_{1}(\Omega A, E) \ar[ldd]
	&
\\
	&
	& 
	&
	& 
	&
	& 
	&
	&	
\\
	&
	& \Omega^{2}A \otimes C \ar[ldd] \ar[rr]
	& 
	& \Omega^{2}A \otimes D \ar[ldd] \ar[rr]
	& 
	& \Omega^{2}A \otimes E \ar[ldd]
	& 
	&
\\
	&
	& 
	&
	& 
	&
	& 
	&
	&	
\\
	& P_{1} \otimes C \ar[ldd] \ar[rr]
	&
	& P_{1} \otimes D \ar[ldd] \ar[rr]
	&
	& P_{1} \otimes E \ar[ldd]
	&
	& 
	&
\\
	\Tor_{1}(A, C) \ar[rr] \ar@{>->}[d]
	&
	& \Tor_{1}( A, D) \ar[rr] \ar@{>->}[d]
	&
	& \Tor_{1}(A, E) \ar@{>->}[d]
	&
	&
	&
	&
\\
	\Omega A \otimes C \ar[rr]^{\gamma} \ar[d]_{\alpha} \ar@{} 	[drr] |{T}
	&
	& \Omega A \otimes D \ar@{->>}[rr] \ar[d]^{\delta}
	&
	& \Omega A \otimes E \ar[d]
	&
	& 
	&
	&
\\
	P_{0} \otimes C \ar@{>->}[rr]_{\beta} \ar[d]
	&
	& P_{0} \otimes D \ar@{->>}[rr] \ar[d]
	&
	& P_{0} \otimes E \ar[d]
	&
	& 
	&
	&
\\
	A \otimes C \ar[rr]
	&
	& A \otimes D \ar@{->>}[rr]
	&
	& A \otimes E 
	&
	& 
	&
	&
	}	
\end{equation*}
\normalsize
Notice that the connecting homomorphism 
$\Tor_{2}(A, E) = \Tor_{1}(\Omega A, E) \overset{\epsilon}{\longrightarrow} \Omega A \otimes C$ in the horizontal part of the diagram factors through 
$\Ker\alpha =\Tor_{1}(A, C)$. Indeed, the commutativity of the square $T$ shows that $\beta \alpha \epsilon = \delta \gamma \epsilon = 0$. Since $\beta$ is monic, $\alpha \epsilon = 0$ and 
$\epsilon$ factors through $\Tor_{1}(A, C)$. As a result, we have a connecting homomorphism $\Tor_{2}(A, E)  \longrightarrow \Tor_{1}(A, C)$ and the desired long exact sequence.

Returning to the left $\Lambda$-module $B$, we specialize the original short exact sequence to the cosyzygy sequence $0 \longrightarrow \Sigma B \longrightarrow I^{1} \longrightarrow \Sigma^{2} B \longrightarrow 0$. The foregoing argument then yields a commutative square 
\[
\xymatrix
	{
	\Tor_{2}(A, \Sigma^{2} B) \ar[r] \ar@{=}[d] \ar[rd]
	& \Tor_{1}(A, \Sigma B) \ar@{>->}[d]
\\
	\Tor_{1}(\Omega A, \Sigma^{2} B) \ar[r] 
	& \Omega A \otimes \Sigma B
	}
\] 
where the diagonal map is the connecting homomorphism in the horizontal part of the diagram on page~\pageref{glue}. The composition of this map with
$\gamma : \Omega A \otimes \Sigma B \to \Omega A \otimes I^{1}$ is zero,  hence it factors through the kernel of $\gamma$, which is by definition 
$\Omega A\, \ot\, \Sigma B$. We now have a commutative diagram of solid arrows
\[
\xymatrix
	{
	\Tor_{2}(A, \Sigma^{2} B) \ar[rr] \ar@{=}[dd] \ar[rd]
	&
	& \Tor_{1}(A, \Sigma B) \ar@{>->}[dd]
	&
\\
	& \Omega A\, \ot \, \Sigma B \ar[rd] \ar@{>.>}[ur]
	&
	&	
\\
	\Tor_{1}(\Omega A, \Sigma^{2} B) \ar[rr] 
	&
	& \Omega A \otimes \Sigma B \ar[r] \ar[d] \ar@{} [dr] |T
	& \Omega A \otimes I^{1} \ar[d]
\\
	&
	& P_{0} \otimes \Sigma B \ar@{>->}[r]
	& P_{0} \otimes I^{1}
	}
\] 
The existence of the dotted arrow and the fact that it is monic was established when we discussed the diagram~\eqref{D:big}; it makes the triangle on the right commute. Since the vertical map in that triangle is monic, it is easy to see that the top triangle is also commutative. By construction, the  horizontal map in that triangle is the connecting homomorphism 
in the long exact sequence of the functors $\Tor_{i}(A, \blank)$ corresponding to the chosen cosyzygy sequence. Colloquially, the connecting homomorphism in the long exact $\Tor$-sequence factors through the injective stabilization. We view these connecting homomorphisms as the structure maps in the sequence
\[
\ldots \longrightarrow \Tor_{1}(\Omega^{i}A, \Sigma^{i+1} B)
\longrightarrow \Tor_{1}(\Omega^{i-1}A, \Sigma^{i} B) \longrightarrow \ldots
\]

On the other hand, as~\eqref{D:big} showed, the structure map
$\Omega A\, \ot\, \Sigma B \longrightarrow A \ot B$
factors through $ \Tor_{1}(A, \Sigma B)$. Combining these two observations, we have a commutative diagram 
\begin{equation}\label{Eq:intertwine}
\begin{gathered}
\xymatrix@C5pt
	{
	\ldots \ar[rr]
	&
	&\Tor_{1}(\Omega A, \Sigma^{2} B) \ar[rr]^{\Gamma_{2}}  \ar@{->>}[rd]
	&
	& \Tor_{1}(A, \Sigma B) \ar@{->>}[rd] \ar[rr]^{\Gamma_{1}}
	&
	& A \otimes B
\\
	& \ldots \ar[rr]^{\Delta_{2}} \ar@{>->}[ru]
	&
	& \Omega A \ot \Sigma B \ar[rr]^{\Delta_{1}} \ar@{>->}[ru]
	&
	& A \ot B \ar@{>->}[ru]
	}
\end{gathered}
\end{equation}
``intertwining'' the two sequences. Taking into account the dimension shift, we now have

\begin{theorem}\label{T:second}
The sequence of the $\Tor$ functors in the above diagram is functorial in the first argument. For any integer $n$, the two families of parallel arrows in the (suitably shifted) above diagram induce mutually inverse isomorphisms\footnote{The reader has probably noticed that, as promised, this theorem implies Proposition~\ref{P:flat-dim}.}
\begin{align*}
 \T_n(A,\blank)(B)  = \underset{k,k+n\geq 0}{\varprojlim} \Omega^{k+n} A\,\ot\ \Sigma^k B  
 \simeq  \underset{k,k+n\geq 0}{\varprojlim} \Tor_{1} (\Omega^{k+n} A, \Sigma^{k+1} B) 
\end{align*}  
\end{theorem} 

\begin{proof}
 The first assertion follows from the functoriality of the connecting homomorphism. The second assertion has already been established.
\end{proof}

\begin{remark}
 We have actually proved more. Since all southeast maps are epic, all northeast maps are monic, and since an epi-mono factorization of a morphism in an abelian category is determined uniquely up to an isomorphism, the lower sequence is determined uniquely up to an isomorphism by the maps in the upper sequence. In particular, this yields new equivalent definitions of both the injective stabilization and the asymptotic stabilization of the tensor product.
\end{remark}

\begin{remark}
 While the top sequence results in a bifunctorial construction, the terms of that sequence are not functorial in the second variable. This is due to the fact the symbol 
 $\Sigma B$ is only defined up to (injective) stable equivalence, but the Tor functor is not injectively stable and thus fails in general to preserve composition when combined with $\Sigma$.
\end{remark}


Returning to the diagram~\eqref{D:big} and using Lemma~\ref {L:gamma}, we have

\begin{lemma}\label{L:delta-epi}
The diagram~\eqref{D:big} gives rise to an exact sequence 
\[
\Tor_1(A, B) \lra \Tor_1(A, I^{0}) \lra \Tor_1(A,\Sigma B) \overset{\delta}\lra A \ot B \lra 0,
 \]
where $\delta$ is (the corestriction of) the connecting homomorphism
from the snake lemma. If the injective $I^{0}$ is flat, then $\delta$ is an isomorphism. \qed
\end{lemma}

\begin{example}\label{E:qF}
 Suppose that $\Lambda$  is quasi-Frobenius or, more generally, left IF (i.e., each injective left module is flat). Then, by Lemma~\ref{L:delta-epi}, the southeast maps are all isomorphisms, making the two systems isomorphic. The next example shows that the two systems may be isomorphic over other types of rings.
\end{example}

\begin{example}
Let $\Lambda := \Z$, $A : = \Z/p\Z$, where $p$ is a prime number, and $B := \Z$. Then, taking $\Sigma B \simeq \Q/\Z$, we have, since the injective stabilization vanishes on injectives (\cite[Lemma~4.5]{MR-1}),
$\Omega A\, \ot\, \Sigma B = 0$.\footnote{Alternatively, since $\Omega A$ is projective, one can use~\cite[Lemma~4.8]{MR-1}).}
To compute $A\, \ot\, B$, we apply the functor $\Z/p\Z \otimes \blank$ to the injective envelope $\Z \to \Q$ of 
$\Z$. The kernel of the resulting map $\Z/p\Z \to \Z/p\Z \otimes \Q$ is $\Z/p\Z$. Finally, we compute $\Tor_{1}(A, \Sigma B)$ by using a projective resolution of $A = \Z/p\Z$. The result is the kernel of the map $\Q/\Z \overset{.p}{\longrightarrow} \Q/\Z$, which is the subgroup $\Z/p\Z = \{0, 1/p, \ldots , (p-1)/p\}$. Moreover, the diagram~\eqref{D:big} shows that, in this case, the map $\Tor_{1}(A, \Sigma B) \to A\, \ot\, B$ is an isomorphism. Since all the remaining terms in the two directed systems vanish, we have that the southeast maps in~\eqref{Eq:intertwine} make the two systems isomorphic. 
\end{example}

\subsection{The connectedness property}

In this subsection we want to show that the second construction of the asymptotic stabilization also produces a connected sequence of functors. This is certainly true since the first and the second constructions produce, as we just saw, isomorphic results.  But we shall show that the second construction has a connecting homomorphism of its own and that it is compatible with the one from the first construction.\footnote{Said differently, under the identification of the two constructions, the connecting homomorphism from the first construction will be identified with the intrinsically defined connecting homomorphism from the second construction.}  

We continue to work with the right $\Lambda$-module $A$ and the short exact sequence 
\[
 0 \longrightarrow B' \longrightarrow B \longrightarrow B'' \longrightarrow 0
\]
of left $\Lambda$-modules.  The approach we are about to describe will make use of the functorial doubly infinite exact sequence~\cite[(9.1)]{MR-1}
\begin{equation}\label{Eq:tor-harpoon}
 \begin{gathered}
 \xymatrix@R10pt
	{
	\ldots \ar[r]
	& \Tor_{1}(A,B') \ar[r]
	& \Tor_{1}(A,B) \ar[r]
	& \Tor_{1}(A,B'')
	&
\\
	\ar[r]
	& A \ot B' \ar[r]
	& A \ot B \ar[r]
	& A \ot B''
	&
\\
	\ar[r]
	& A \ot \Sigma B' \ar[r]
	& A \ot \Sigma B \ar[r]
	& A \ot \Sigma B'' \ar[r]
	& \ldots
	}
\end{gathered}
\end{equation}
In this sequence, each injective stabilization is part of a functorial directed sequence; this was established in Proposition~\ref{P:functorial}. In particular, 
%
%
the structure maps $\Delta_{i}$ are functorial in the second argument. This implies that  \texttt{each row of the injective stabilizations} 
in~\eqref{Eq:tor-harpoon}
gives rise to morphisms of the corresponding directed systems. 


Now we want to build structure maps for each of the Tor terms 
in~\eqref{Eq:tor-harpoon}
In fact, those maps have already been built in the second construction of the asymptotic stabilization, see the diagram~\eqref{Eq:intertwine} and the diagram on page~\pageref{glue}, where $B$ (respectively, $B'$, $B''$) should be replaced by $\Sigma B$ (respectively, $\Sigma B'$, $\Sigma B''$). As Theorem~\ref{T:second} shows, the resulting system is functorial in each argument. Therefore, \texttt{each row of the Tor functors} in~\eqref{Eq:tor-harpoon}
gives rise to morphisms of the corresponding directed systems.


Now we claim that the term-wise maps between the directed systems constitute morphisms of those systems. The foregoing discussion shows that we only have to check the commutativity of the squares lying over the connecting homomorphisms in~\eqref{Eq:tor-harpoon}.
There are three types of such homomorphisms: between two copies of $\Tor$, between $\Tor$ and the injective stabilization, and between two copies of the injective stabilization. It is clear that we only have to check one square of each type. We first examine the square(s) connecting $\Tor_{1}$ and the corresponding injective stabilization. 

\begin{lemma}\label{L:tor-stab}
The square
 \[
\xymatrix
	{
	\Tor_{1}(\Omega A, \Sigma B'') \ar[d]_{\Gamma_{2}} \ar[r]^-{h}
	& \Omega A \ot \Sigma B' \ar[d]^{\Delta_{1}}
\\
	\Tor_{1}(A, B'') \ar[r]^-{g}
	& A \ot B'
	}
\] 
anti-commutes. Here $\Gamma_{2}$ is from~\eqref{Eq:intertwine} (with $\Sigma B$ replaced by $B''$),  $\Delta_{1}$ is from~\eqref{Eq:directed} (with $B$ replaced by $B'$), and the horizontal maps are the connecting homomorphisms from~\eqref{Eq:tor-harpoon}.
\end{lemma}

\begin{proof}
In terms of the commutative 3D diagram below, $\Gamma_{2}$ is determined by the connecting homomorphism in the front face, with domain a subgroup of the framed term. Similarly, $h$ is determined by the connecting homomorphism in the right-hand face of the cube, with its image being a subgroup of the framed node in the back.
%
\[
\resizebox{1.0\textwidth}{!}{
\xymatrix@R25pt@C35pt
	{
	&
	& \Omega^{2} A \otimes B' \ar[rr] \ar[d] \ar[lddd]
	&
	& \Omega^{2} A \otimes I^{0'} \ar@{->>}[rr] \ar[d] \ar[lddd]
	&
	& \Omega^{2} A \otimes \Sigma B' \ar[d] \ar[lddd]
\\
	&
	& P_{1} \otimes B' 
	\ar[rr] |! {[urr];[ddr]}\hole \ar@{->>}[d] 		
	\ar[lddd] |! {[ddl];[ddr]}\hole
	&
	& P_{1} \otimes I^{0'} 
	\ar@{->>}[rr] |!{[urr];[ddr]}\hole \ar@{->>}[d] 
	\ar[lddd] |! {[ddl];[ddr]}\hole
	&
	& P_{1} \otimes \Sigma B' 
	\ar@[blue]@{->>}[d] 				
	\ar@[blue]@{>->}[lddd]
\\
	&
	& \Omega A \otimes B' 
	\ar[rr] |! {[uurr];[dr]}\hole |!{[urr];[ddr]}\hole 
	\ar[lddd] |! {[dl];[dr]} \hole |! {[ddl];[ddr]} \hole  
	&
	& \Omega A \otimes I^{0'} 
	\ar@[magenta]@{->>}[rr] |! {[uurr];[dr]}\hole |!{[urr];[ddr]} \hole 
	\ar@[magenta]@{>->}[lddd] |!{[dl];[dr]}\hole |!{[ddl];[ddr]} \hole
	&
	& *+[F]{\Omega A \otimes \Sigma B'} \ar[lddd]
\\
	& \Omega^{2} A \otimes B \ar[rr]  \ar[d] \ar@{->>}[lddd]
	&
	& \Omega^{2} A \otimes I^{0} 
	\ar@{->>}[rr] \ar[d] \ar@{->>}[lddd]
	&
	& \Omega^{2} A \otimes \Sigma B \ar@[blue][d] 			\ar@[blue]@{->>}[lddd]
	&
\\
	& P_{1} \otimes B 
	\ar[rr] |! {[urr];[ddr]}\hole \ar@{->>}[d] 
	\ar@{->>}[lddd] |! {[ddl];[ddr]}\hole
	&
	& P_{1} \otimes I^{0} 
	\ar@{->>}[rr] |! {[urr];[ddr]}\hole \ar@{->>}[d] 
	\ar@{->>}[lddd] |! {[ddl];[ddr]}\hole
	&
	& P_{1} \otimes \Sigma B \ar@{->>}[d] \ar@{->>}[lddd]
	&
\\
	& \Omega A \otimes B 
	\ar@[magenta][rr] |! {[uurr];[dr]}\hole |! {[urr];[ddr]}\hole 
	\ar@[magenta]@{->>}[lddd] |!{[dl];[dr]}\hole |! {[ddl];[ddr]}\hole
	&
	&  \Omega A \otimes I^{0} 
	\ar@{->>}[rr] |! {[uurr];[dr]}\hole  |! {[urr];[ddr]}\hole 
	\ar@{->>}[lddd] |! {[dl];[dr]} \hole |! {[ddl];[ddr]}\hole
	&
	& \Omega A \otimes \Sigma B \ar@{->>}[lddd]
	&
\\
	\Omega^{2} A \otimes B'' \ar[rr] \ar[d]
	&
	& \Omega^{2} A \otimes I^{0''} \ar@[red]@{->>}[rr] \ar@[red]	[d]
	&
	& *+[F]{\Omega^{2} A \otimes \Sigma B''} \ar[d]
	&
	&
\\
	P_{1} \otimes B'' \ar@[red]@{>->}[rr]  \ar@[red]@{->>}[d]
	&
	& P_{1} \otimes I^{0''}  \ar@{->>}[rr] \ar@{->>}[d]
	&
	& P_{1} \otimes \Sigma B'' \ar@{->>}[d]
	&
	&
\\
	 \Omega A \otimes B'' \ar[rr] 
	&
	& \Omega A \otimes I^{0''} \ar@{->>}[rr] 
	&
	& \Omega A \otimes \Sigma B''
	&
	&
	}
	}
\] 

The next commutative diagram shows that $g$ is determined by the connecting homomorphism in the left-hand face and starts from inside the framed term. It also shows that $\Delta_{1}$ is determined by the connecting homomorphism in the back face of the cube, with its image inside the framed term in the back.
 
 \[
 \resizebox{1.0\textwidth}{!}{ 
\xymatrix@R25pt@C35pt
	{
	&
	& \Omega A \otimes B' \ar[rr] \ar[d] \ar[lddd]
	&
	& \Omega A \otimes I^{0'} \ar@[purple]@{->>}[rr] 
	\ar@[red][d] 
	\ar@[magenta]@{>->}[lddd] 
	&
	& \Omega A \otimes \Sigma B' \ar[d] \ar[lddd] 
\\
	&
	& P_{0} \otimes B' \ar@[red]@{>->}[rr] |! {[urr];[ddr]}\hole 
	\ar@[purple]@{->>}[d] 
	\ar@[blue]@{>->}[lddd] |! {[ddl];[ddr]}\hole
	&
	& P_{0} \otimes I^{0'} \ar@{->>}[rr]  |!{[urr];[ddr]}\hole 
	\ar@{->>}[d] 
	\ar@{>->}[lddd] |! {[ddl];[ddr]}\hole
	&
	& P_{0} \otimes \Sigma B' \ar@{->>}[d] \ar[lddd]
\\
	&
	& *+[F]{A \otimes B'} 
	\ar[rr] |! {[uurr];[dr]}\hole |!{[urr];[ddr]}\hole 
	\ar[lddd] |! {[dl];[dr]} \hole |! {[ddl];[ddr]} \hole  
	&
	& A \otimes I^{0'} 
	\ar@{->>}[rr] |! {[uurr];[dr]}\hole |!{[urr]; [ddr]} \hole 
	\ar[lddd] |! {[dl];[dr]} \hole |! {[ddl];[ddr]} \hole
	&
	& A \otimes \Sigma B' \ar[lddd]
\\
	& \Omega A \otimes B \ar@[magenta][rr] 
	\ar@[blue][d]  
	\ar@[purple]@{->>}[lddd]  
	&
	& \Omega A \otimes I^{0} \ar@{->>}[rr] \ar[d]  
	\ar@{->>}[lddd]
	&
	& \Omega A \otimes \Sigma B \ar[d] \ar@{->>}[lddd]
	&
\\
	& P_{0} \otimes B \ar@{>->}[rr] |! {[urr];[ddr]}\hole 
	\ar@{->>}[d] 
	\ar@{->>}[lddd] |! {[ddl];[ddr]}\hole
	&
	& P_{0} \otimes I^{0} \ar@{->>}[rr] |! {[urr];[ddr]}\hole 
	\ar@{->>}[d] 
	\ar@{->>}[lddd]  |! {[ddl];[ddr]}\hole
	&
	& P_{0} \otimes \Sigma B \ar@{->>}[d] \ar@{->>}[lddd]
	&
\\
	& A \otimes B \ar[rr] |! {[uurr];[dr]}\hole |! {[urr];[ddr]}\hole 
	\ar@{->>}[lddd] |! {[dl];[dr]} \hole |! {[ddl];[ddr]}\hole
	&
	& A \otimes I^{0} 
	\ar@{->>}[rr] |! {[uurr];[dr]}\hole  |! {[urr];[ddr]}\hole 
	\ar@{->>}[lddd] |! {[dl];[dr]} \hole |! {[ddl];[ddr]}\hole
	&
	& A \otimes \Sigma B \ar@{->>}[lddd]
	&
\\
	*+[F]{\Omega A \otimes B''} \ar[rr]  \ar[d] 
	&
	& \Omega A \otimes I^{0''} \ar@{->>}[rr] \ar[d]
	&
	& \Omega A \otimes \Sigma B'' \ar[d]
	&
	&
\\
	P_{0} \otimes B'' \ar[rr]  \ar@{->>}[d]
	&
	& P_{0} \otimes I^{0''} \ar@{->>}[rr] \ar@{->>}[d]
	&
	& P_{0} \otimes \Sigma B'' \ar@{->>}[d]
	&
	&
\\
	A \otimes B'' \ar[rr]
	&
	& A \otimes I^{0''} \ar@{->>}[rr] 
	&
	& A \otimes \Sigma B''
	&
	&
	}
	}
\] 
 

One can easily check that the first cube satisfies the conditions of Lemma~\ref{L:down-horizontal}, and the second cube satisfies the conditions of Lemma~\ref{L:horizontal-down}. However, we cannot immediately apply these results because neither $g\Gamma_{2}$ nor $\Delta_{1}h$ is contained in a single cube. To bypass this obstacle, notice that the bottom face of the first cube coincides with the top face of the second cube. Let $\delta$ be the connecting homomorphism in this common face (it starts inside $\Omega A \otimes B''$ and ends inside $\Omega A \otimes \Sigma B'$). By Lemma~\ref{L:down-horizontal}, $-h = \delta \Gamma_{2}$ and, by Lemma~\ref{L:horizontal-down}, $g = \Delta_{1} \delta$. Therefore $g\Gamma_{2} = - \Delta_{1}h$, as claimed. 
\end{proof}

Now we look at the square(s) connecting two consecutive Tor functors.

\begin{lemma}
 In the above notation\footnote{To keep the notation simple, the same subscripted $\Gamma$'s denote the structure maps relative to $B'$ and $B''$.}, the square
 \[
\xymatrix
	{
	\Tor_{2}(\Omega A, \Sigma B'') \ar[d]_{\Gamma_{3}} \ar[r]^-{h}
	& \Tor_{1}(\Omega A, \Sigma B') \ar[d]^{\Gamma_{2}}
\\
	\Tor_{2}(A, B'') \ar[r]^-{g}
	& \Tor_{1}(A, B')
	}
\] 
anti-commutes.
\end{lemma}

\begin{proof}
 The argument in this case is identical to that of the previous lemma, except that, in the two cubes, $A$ has to be replaced by 
 $\Omega A$. The details are left to the reader.
\end{proof}

Finally, we examine the square(s) connecting two consecutive shifts of the injective stabilizations.

\begin{lemma}
 In the above notation\footnote{To keep the notation simple, we ``recycle'' the symbol $\Delta_{1}$.}, the square
 \[
\xymatrix
	{
	\Omega A \ot \Sigma B'' \ar[d]_{\Delta_{1}} \ar[r]^-{h}
	& \Omega A \ot \Sigma^{2} B' \ar[d]^{\Delta_{1}}
\\
	A \ot B'' \ar[r]^-{g}
	& A \ot \Sigma B'
	}
\] 
anti-commutes.
\end{lemma}

\begin{proof}
 The argument is similar to those in the previous two lemmas. To describe the composition $g\Delta_{1}$ we use the first cube from the proof of Lemma~\ref{L:tor-stab}, where we lower the index of each copy of $\Omega$ by one. Let $\delta$ be the connecting homomorphism in the right-hand face of that cube. By Lemma~\ref{L:down-horizontal}, $g\Delta_{1} = - \delta$. 
 
To describe the composition $\Delta_{1}h$ we use the second cube from the proof of Lemma~\ref{L:tor-stab}, where we raise the index of each copy of $\Sigma$ by one (in particular, $B$ becomes 
 $\Sigma B$, etc.). The left-hand face of this cube coincides with the right-hand face of the previous cube, so they share the connecting homomorphism $\delta$. By Lemma~\ref{L:horizontal-down}, $\delta = \Delta_{1}h$, and therefore $g\Delta_{1} = - \Delta_{1}h$, as claimed.
\end{proof}

The just proved results show that, to obtain morphisms of the directed systems, we have to offset the sign in the squares that contain connecting homomorphisms. For example, we can leave the vertical directed systems unchanged, but modify the horizontal long exact sequences by introducing an alternating (in the vertical direction) sign for the connecting homomorphisms. In summary, we have

\begin{theorem}
 The pair  $(\T_\bullet(A,\blank), \rho_{\bullet})$, where 
 $\rho_{\bullet}$ is the limit of the connecting homomorphisms modified as above, is a connected sequence of functors. It is isomorphic to the connected sequence $(\T_\bullet(A,\blank), \omega_{\bullet})$ of Theorem~\ref{T:first-conn-sec}.
\end{theorem}
\begin{proof}
 The first assertion has just been proved. The verification of the second is straightforward and is left to the reader.
\end{proof}

\section{The asymptotic stabilization: the third construction}\label{S:third-const}

The goal of this section is to establish an isomorphism between the asymptotic stabilization of the tensor product and the $J$-completion of the univariate $\Tor$ functor.

\subsection{Stabilization via satellites}
The functor isomorphism $S^{1}\Tor_{1}(A, \blank) \cong A \ot \blank$ \cite[Proposition~9.3]{MR-1}
 ($S^{1}$ denotes the \textcolor{purple}{right} satellite) 
suggests a stabilization sequence of the form
\[
\Delta_{i} : S^{1}\Tor_{1}(\Omega^{i+1}A, \blank) \circ \Sigma^{i+1} \longrightarrow
S^{1}\Tor_{1}(\Omega^{i}A, \blank) \circ \Sigma^{i}.
\]
Clearly, it suffices to do this for $i=0$, and we shall again use the 
diagram~\eqref{D:big}. This yields a commutative diagram of solid arrows
\begin{equation*}\label{double-big}
\begin{gathered}
 \xymatrix@R20pt@C8pt
 	{
	&
	&
	& \Tor_{1}(\Omega A, I^{1}) \ar[dr]^-{\epsilon_{1}} 
	& 
	& 0 \ar[d]
	&	
\\
	&
	&
	&
	&\Tor_{1}(\Omega A, \Sigma^{2} B) \ar[dr]^-{\alpha_{1}} \ar@{.>}[r]
	& \Tor_1(A,\Sigma B) \ar[d]
	&
\\
	&
	&
	& \Omega A\otimes B \ar[d] \ar[r]
	& \Omega A\otimes I^0 \ar[d] \ar[r]
	& \Omega A\otimes 	\Sigma B \ar[d] \ar[r] \ar[dr]
	& 0	
\\
	&
	& 0\ar[r]
	& P_0\otimes B \ar[r] \ar[d] 
	& P_0\otimes I^0 \ar[r] \ar[d] 
	& P_0\otimes \Sigma B \ar[d] \ar@{>->}[dr]
	& \Omega A\otimes I^1 \ar[d]
\\
	&\Tor_{1}(A, I^{0}) \ar[r]^-{\epsilon}
	& \Tor_1(A,\Sigma B) \ar[r]^-{\alpha}
	& A\otimes B \ar[d] \ar[r]
	& A\otimes I^0 \ar[r] \ar[d]
	& A\otimes \Sigma B \ar[d]
	& P_{0} \otimes I^{1}
\\
	&
	&
	& 0 
	& 0
	& 0
	&
	} 
\end{gathered}
\end{equation*}
with exact rows and columns. Moreover, the diagonal is a fragment of a long homology exact sequence and, at the same time, the bottom row of~\eqref{D:big} with the fixed argument specialized to $\Omega A$ and with $\Sigma B$ replaced by $\Sigma^{2} B$ and $I^{0}$ replaced by $I^{1}$. The diagram shows that $\alpha_{1}$ factors through $\Tor_{1}(A, \Sigma B)$, giving rise to a unique dotted map making a commutative triangle. As 
$\Tor_1(A,\Sigma B) \longrightarrow \Omega A \otimes \Sigma B$ is monic, the dotted map composed with $\epsilon_{1}$ is zero, and therefore gives rise to a unique map
\[
\Coker \epsilon_{1} = S^{1}\Tor_{1}(\Omega A, \blank)(\Sigma B)
\longrightarrow \Tor_1(A,\Sigma B)
\]
Composing it with the isomorphism $ \gamma : \Tor_1(A, \Sigma B) \longrightarrow \Tor_1(A, \Sigma B)$ constructed in Lemma~\ref{L:gamma}
and with the canonical epimorphism 
\[
\Tor_1(A,\Sigma B) \longrightarrow 
\Coker \epsilon = S^{1}\Tor_{1}(A, \blank)(B),
\]
we declare the resulting composition to be the structure map 
\[
\Delta_{1} : S^{1}\Tor_{1}(\Omega A, \blank)(\Sigma B) \longrightarrow
S^{1}\Tor_{1}(A, \blank)(B).
\]
By Lemma~\ref{L:gamma}, it is compatible with the isomorphisms 
$S^{1}\Tor_{1}(A, \blank) \overset{\cong}\lra A \ot \blank$
(for fixed $A$ and $\Omega A$). Similar  arguments yield maps 
$\Delta_{i}$ for all natural $i$ and the compatibility with the connecting homomorphisms. We now have 
\begin{theorem}~\label{T:Jcompdir}
 The connecting homomorphism in the diagram~\eqref{D:big} induces an isomorphism of connected sequences of functors
\[
(S^{1} \Tor_{1}(\Omega^{i}A, \blank) \circ \Sigma^{i}, \Delta_{i}) 
\simeq 
(\Omega^{i} A \, \ot \, \Sigma^{i} \blank, \Delta_{i})
\] 
natural in $A$. \qed
\end{theorem} 

Passing to the componentwise limits in the foregoing isomorphisms, we have, in summary, that the three constructions of the asymptotic stabilization of the tensor product yield isomorphic connected sequences of functors.

\begin{corollary}\label{C:T-vs-J}
The asymptotic stabilization $\T(A,\blank)$ and the $J$-completion of the connected sequence $\Tor_{\ast}(A,\blank)$ are isomorphic as connected sequences of functors.\end{corollary}

\begin{proof}This follows from the fact that the directed system involved in the construction of the $J$-completion~\cite{Tr} and the directed system used in the construction of the asymptotic stabilization are isomorphic.  The isomorphism is precisely that appearing in Theorem ~\ref{T:Jcompdir}. \end{proof}
\begin{remark}
 For a more precise statement of the foregoing corollary see~Proposition~\ref{P:comparison} below.
\end{remark}

\section{The comparison homomorphisms}\label{S:comparison}

At the moment we have three constructions of stable homology: Vogel homology, the $J$-completion of the univariate $\Tor$ functors, and the asymptotic stabilization of the tensor product. Our next goal is to compare them.

\subsection{From Vogel homology to the asymptotic stabilization of the tensor product}

First, we want to construct a natural transformation from Vogel homology to the asymptotic stabilization of the tensor product. This will be done in degree zero;  all other degrees are treated similarly. Let $A$ be a right $\Lambda$-module with a projective resolution $(P, \partial_{P}) \lra A$, and $B$ be a left $\Lambda$-module with an injective resolution $B \lra (I, \partial_{I})$. Recall that the differential on $\V_{\bullet} (A,B)$ is induced by  
$ \partial_{P} \otimes 1 +(-1)^{\deg_{1}(\blank)}1\otimes \partial_{I}$ 
(see~\eqref{vogel-differential}).
To simplify notation, we set 
\[
d_{P} := \partial_{P} \otimes 1 \quad \text{and} \quad d_{I} := 1\otimes \partial_{I}.
\]
A homology class in $\V_0 (A,B)$ can be represented by an infinite sequence 
\[
s = (s_i)_{i=1}^\te \in (P_1\otimes I^0) \times (P_2\otimes I^1)\times \cdots
\] 
which vanishes under the differential of $\V_\bullet(A,B)$. This means that  
\[
D(s)=(d_P(s_1), -d_I(s_1)+d_P(s_2), d_I(s_2)+d_P(s_3),-d_I(s_3)+d_P(s_4),\ldots)
\]
represents the zero class in $\V_{-1}(A,B)$ and therefore has only finitely many nonzero components. Let $k$ be the smallest index such that 
\begin{equation}\label{confluence}
 \begin{aligned}
 d_I(s_k) 		& = d_P(s_{k+1}) \\
 d_I(s_{k+1}) 	& = -d_P(s_{k+2}) \\
 d_I(s_{k+2}) 	& = d_P(s_{k+3}) \\
 d_I(s_{k+3})	& =-d_P(s_{k+4}) \\
 			& \ldots
\end{aligned}
\end{equation}
In short, for all $i \geq 0$ 
\[
d_{I}(s_{k+i})=(-1)^{i} d_{P}(s_{k+i+1})
\]
Observe that since $s_{k+1}\in P_{k+1}\otimes I^k$, $d_P(s_{k+1})\in P_k\otimes I^k$.  Denote $d_P(s_{k+1})$ by $\bullet$ in the following commutative diagram with exact rows and columns:
\dia
	{
	& \Omega^{k+1}A \otimes \Sigma^{k}B \ar[d] \ar[r]
	& \underset{\circ}{\Omega^{k+1} A \otimes I^k} \ar[d] \ar[r]
	& \Omega^{k+1}A\otimes \Sigma^{k+1}B\ar[d] \ar[r]
	& 0 
\\
	0 \ar[r]
	&\ar[d]\underset{\Box}{P_k\otimes \Sigma^k B}\ar[r]
	& \ar[d]\underset{\bullet}{P_k\otimes I^k}\ar[r]
	& \ar[d]P_k\otimes \Sigma^{k+1} B \ar[r]
	& 0
\\
	& \ar[d]{\Omega^k A\otimes \Sigma^k B}\ar[r]
	&\ar[d]\Omega^kA\otimes I^k\ar[r]
	&\ar[d]\Omega^k A\otimes \Sigma^{k+1}B
\\
	& 0
	& 0
	& 0
	}  
Since $\bullet = d_{I}(s_{k})$, it pulls back to some element 
$\Box$. Pushing it down, we produce $\omega_{k} \in \Omega^k A \otimes \Sigma^k B$. Since $\bullet = d_{P}(s_{k + 1})$, the element $\bullet$ is the image of some element $\circ$ in $\O^{k+1}A \otimes I^{k}$. By the commutativity of the diagram, the image of $\omega_{k}$ in $\O^{k} A \otimes I^{k}$ is zero, i.e., $\omega_{k} \in \O^k A \ot \s^k B$, and we set 
$\varphi_{k} : = \omega_{k}$.

This process is well-defined up to choice of sign.  To see this, notice that the element $\bullet$ goes to 0 when applying the horizontal map.  Hence, by commutativity of the diagram, $s_{k+1}$ also goes to 0 using the vertical top right map and hence is in the kernel of this map.   Now one can apply the map from the snake lemma to produce the exact same element $\omega_{k}$.  Since we may also take the negative of this connecting homomorphism, we even have the freedom to choose $\pm \omega_{k}$.  Once this choice is fixed, $\omega_{k}$ will be well-defined.  

 
Next, apply the same procedure to $d_P(s_{k+2}) = -d_I(s_{k+1}) \in P_{k+1} \otimes I^{k+1}$, producing $\omega_{k+1} \in \O^{k+1} A \ot \s^{k+1} B$. Flipping its sign, we set $\varphi_{k+1} := - \omega_{k+1}$. To obtain 
$\varphi_{k+2}$, perform the same procedure with $d_P(s_{k+3})$ and set 
$\varphi_{k+2} := - \omega_{k+2}$. To obtain $\varphi_{k+3}$, perform the same procedure with $d_P(s_{k+4})$ and set $\varphi_{k+3} := \omega_{k+3}$.  
Iterating this process, for any $ i \geq 0$, we set 
\[
\varphi_{k + i} := 
\begin{cases}
 \omega_{k +i} & \text{if } i \equiv 0,3 \pmod {4} \\
 -\omega_{k +i} & \text{if } i \equiv 1,2 \pmod {4}.
\end{cases}
\]
We claim that the sequence $(\varphi_k, \varphi_{k+1},\ldots)$ is coherent, i.e., in the notation of~\eqref{Eq:directed}, $\Delta_{n}(\varphi_{n}) = \varphi_{n-1}$ for any $n \geq k +1$. It suffices to check this claim for $n = k +1$; the remaining cases are similar. To this end, we examine the commutative diagram 
\[
\xymatrix
	{
	& P_{k + 1} \otimes I^{k} \ar[d] |*+[o][F-]{6} \ar[r] |(.4)*+[o][F-]{5} 
	\ar@{}[dr] |*+[o]	[F-] {T} \ar@/^2pc/[rr] |-{d_{I}}  \ar@/_3pc/[dd] |-{d_{P}}
	& P_{k + 1} \otimes \Sigma^{k + 1} B \ar[d] |*+[o][F-]{2} \ar[r] |*+[o][F-]{1}
	& \underset{\bullet}{P_{k + 1} \otimes I^{k + 1}} \ar[d]
\\
	& \O^{k +1} A \otimes I^{k} \ar[d] |*+[o][F-]{8} \ar[r] |<<<<*+[o][F-]{7}
	& \underset{\varphi_{k+1}}{\O^{k+1} A \otimes \Sigma^{k + 1} B} \ar[d] \ar[r]
	& \O^{k+1} A \otimes I^{k + 1}	
\\
	P_{k} \otimes \Sigma^{k} B \ar[d] |*+[o][F-]{4} \ar[r] |*+[o][F-]{3}
	& \underset{\bullet}{P_{k} \otimes I^{k}} \ar[r]
	& P_{k} \otimes \Sigma^{k + 1} B
\\
	\underset{\varphi_{k}}{\O^{k}A \otimes \Sigma^{k} B}
	}
\] 
Here $(1) \circ (5) = d_{I}$, $(8) \circ (6) = d_{P}$,  and the bullets denote 
\[
d_{P}(s_{k + 2}) = - d_{I} (s_{k + 1}) = - (1) \circ (5) (s_{k + 1}) \quad \text{and} 
\quad d_{P}(s_{k + 1}).
\]
The element $\varphi_{k+1}$ is obtained from the upper bullet by applying 
$(2) \circ (1)^{-1}$ and $\varphi_{k}$ is obtained from the lower bullet by applying $(4) \circ (3)^{-1}$. Since the square $T$ commutes, 
\[
\varphi_{k+1} = - (2) \circ (1)^{-1}(d_{P}(s_{k + 2})) 
= - (7) \circ (6) \circ (5)^{-1} \circ(1)^{-1}(d_{P}(s_{k + 2})).
\]
On the other hand, recalling the construction of $\Delta_{k+1}$ (this is just the restriction of the connecting homomorphism in the diagram~\eqref{D:big}) we have 
\begin{align*}
\begin{split}
 \Delta_{k+1} (\varphi_{k+1}) 
 	& = (4) \circ (3)^{-1}\circ (8) \circ (7)^{-1} (\varphi_{k+1})\\
 	& = - (4) \circ (3)^{-1}\circ (8) \circ (7)^{-1} \circ (7) \circ (6) \circ (5)^{-1} 		\circ(1)^{-1}(d_{P}(s_{k + 2})) \\
	& = (4) \circ (3)^{-1}\circ (8) \circ (6) \circ (5)^{-1} \circ(1)^{-1}
	(d_{I}(s_{k + 1})) \\
	& = (4) \circ (3)^{-1} \circ (8) \circ (6) (s_{k + 1}) \\
	& = (4) \circ (3)^{-1} (d_{P}(s_{k+1})) \\
	& = \varphi_{k}.
\end{split}
\end{align*}
Thus we have shown that the sequence $(\varphi_k, \varphi_{k+1},\ldots)$ is coherent. It uniquely extends to a coherent sequence $(\varphi_i)_{i=0}^{\infty}$,
and we set $\kappa_{0}(s) := (\varphi_i)_{i=0}^{\infty}$. A similar argument yields $\kappa_l : \V_{l}(A,\blank) \lra \T_{l}(A,\blank)$ for each integer $l$. 

\begin{theorem}
Let $A$ be a right $\Lambda$-module. For each $l \in \Z$, 
\[
\kappa_l : \V_{l}(A,\blank) \lra \T_{l}(A,\blank)
\]
is a natural transformation.
\end{theorem}

\begin{proof}
We only need to show the naturality of each $\kappa_{l}$. But this follows from the naturality of the connecting homomorphism.
\end{proof}

\begin{theorem}
In the above notation, for each $l \in \Z$, the natural transformation $\kappa_l : \V_{l}(A,\blank) \lra \T_{l}(A,\blank)$ is an epimorphism.
\end{theorem}

\begin{proof}The proof is primarily a diagram chase and only a sketch will be given.  Let $(\varphi_0,\varphi_1,\varphi_2,\ldots)$ be a coherent sequence in the asymptotic stabilization of the tensor product.  Then $\varphi_k\in \Omega^kA\otimes\Sigma^k B$.  We will construct an element in $V_0(A,B)$ which maps onto this coherent sequence.  One will benefit from the following diagram:  
\dia{&&&\ar[ddl]P_2\otimes I^0\ar@{>>}[d]\ar[r]&P_2\otimes I^1\ar@{>>}[d]\ar[r]&P_2\otimes \s^2 B\ar@{>>}[d]\ar[r]&0\\
&&&\O^{2}A\otimes\Sigma^1B\ar@{-->}[r]\ar@{-->}[d]&\O^2A \otimes I^1\ar@{-->}[d]\ar@{-->}[r]&\underset{\varphi_{_2}}{\O^2A\otimes \s^2B}\ar@{-->}[d]\ar@{-->}[r]&0\\
0\ar[r]&P_1\otimes B\ar[r]\ar@{>>}[d]&P_1\otimes I^0\ar@{>>}[r]\ar@{>>}[d]&P_1\otimes \Sigma^1B\ar@{-->>}[d]\ar@{^{(}-->}[r]&P_1\otimes I^1\ar@{-->}[r]\ar@{-->}[d]&P_1\otimes \s^2 B\ar@{-->}[d]\ar@{-->}[r]&0\\
&\O^1A\otimes B\ar@{=>}[d]\ar@{=>}[r]&\O^1 A\otimes I^0\ar@{=>}[d]\ar@{=>>}[r]&\underset{\varphi_{_1}}{\O^1A\otimes \s^1} B\ar@{=>}[d]\ar@{-->}[r]&\O^1 A\otimes I^1
\ar@{-->}[r]&\O^1 A\otimes \s^2  B\ar@{-->}[r]&0\\
0\ar@{=>}[r]&\ar@{=>}[d]P_0\otimes B\ar@{=>}[r]&\ar@{=>}[d]P_0\otimes I^0\ar@{=>}[r]&\ar@{=>}[d]P_0\otimes \s^1 B\\
&\underset{\varphi_{_0}}{A\otimes B}\ar@{=>}[d]\ar@{=>}[r]& A\otimes I^0\ar@{=>}[d]\ar@{=>}[r]& A\otimes \s B\ar@{=>}[d]\\
&0&0&0
}

Start by selecting $s_1\in P_1\otimes I^0$ that maps onto $\varphi_1$.  Then $d_P(s_1)$ will pullback to $\varphi_0$.  Now select $t_2\in P_2\otimes I^1$ that maps onto $\varphi_2$.  By diagram chase we get that there exists $y_2\in P_2\otimes I^0$ such that $d_I(s_1)-d_P(t_2)=d_P(d_I(y))$ which yields $d_I(s_1)=d_P(t_2-d_I(y))$.  Define $s_2:=t_2-d_I(y_2)$.  Then $s_2$ still maps onto $\varphi_2$ and $d_P(s_2)$ pulls back to $\varphi_1$.  

Now select $t_3\in P_3\otimes I^2$ that maps onto $-\varphi_3$.  Then $-d_P(t_3)$ pulls back to $\varphi_2$ as does $d_I(s_2)$.  By diagram chasing, there exists $y_3\in P_3\otimes I^1$ such that $d_I(s_2)+d_P(t_3)=d_P(d_I(y_3))$.  Define $s_3:=t_3-d_I(y_3)$.  Then $s_3$ maps onto $-\varphi_3$ so $-d_P(s_3)$ pulls back to $\varphi_2$.  Moreover $d_I(s_2)=-d_P(s_3)$.  

If we continue this process paying attention to signs, we can construct an element $(s_k)_{k=1}^\te\in V_0(A,B)$ that maps onto the coherent sequence $(\varphi_k)_{k=1}^\te$.  The details are left to the reader.    \end{proof}

\subsection{From the asymptotic stabilization of the tensor product to the $J$-completion of the univariate $\Tor$}
Let $U$ be a connected sequence of functors and $\M_\bullet(U)$ its $J$-completion (see~\ref{S:J-completion}).  In~\cite[ Proposition 6.1.2]{Tr}, Triulzi shows that there is a morphism of connected sequences of functors 
$\tau:\M_\bullet(U)\to U$ satisfying the following universal property.  Given any morphism $\beta:V\to U$, where~$V$ is a connected sequence of functors that is injectively stable in all degrees, there exists a unique morphism 
$\phi:V\to \M_\bullet(U)$ such that $\phi\tau=\beta$.  From this, we can now establish a commutative diagram of comparison maps between Vogel homology, the asymptotic stabilization of the tensor product, and the $J$-completion of the univariate $\Tor$ functor.

\begin{proposition}\label{P:comparison}
For any module $A$, there is a commutative diagram of connected sequences of functors
\dia
	{
	\V_\bullet(A,\blank) \ar@{->>}[r]^{\kappa} \ar[d]_{\theta}
	& \T_\bullet(A,\blank) \ar[ld]_{\simeq}^{\ge} \ar[d]^{\lambda}
\\
	\M_\bullet(\Tor(A,\blank))\ar[r]_{\tau}
	& \Tor(A,\blank)
	}
where $\tau$ is the $J$-completion of $\Tor(A,\blank)$. Moreover, $\gl \gk$ is the canonical natural transformation from Vogel homology to $\Tor$.
\end{proposition}

\begin{proof}
For $\lambda$, take the natural transformation from the limit to the first term. The diagonal isomorphism $\ge$ is taken from Corollary~\ref{C:T-vs-J}. Under that isomorphism, $\lambda$ is identified with $\gt$, i.e., the lower triangle commutes. Now set $\theta : = \ge \kappa$, thus making the whole square commute. The last assertion is verified by a direction calculation.
\end{proof}

Since the connected sequence of functors $\V_\bullet(A,\blank)$ is $J$-complete, the universal property of the $J$-completion yields

\begin{corollary}\label{C:comparison}
$\theta$ is the unique lifting of the canonical natural transformation 
$\gl \gk : \V_\bullet(A,\blank)\to \Tor(A,\blank)$ against~$\gt$. \qed
\end{corollary}

As an immediate application of Proposition~\ref{P:comparison} and Corollary~\ref{C:comparison}, we have (see also~\cite[Corollary 6.2.10]{Tr} and ~\cite[Theorem 69]{R-Thesis})
\begin{proposition}
 The comparison map from Vogel homology to the $J$-completion of $\Tor$
 is epic in each degree. \qed
\end{proposition} 

In view of the foregoing result, it is natural to try and identify the kernel of 
$\kappa : \V_\bullet(A,\blank) \lra  \T_\bullet(A,\blank)$ or, equivalently, of 
$\theta : \V_\bullet(A,\blank) \lra  \M_\bullet(\Tor(A,\blank))$. Driven by a formal analogy between $\kappa$ and the natural transformation from 
Steenrod-Sitnikov homology to \v{C}ech homology, the first author conjectured in 2014 that the kernel of $\kappa$ should be given by a derived limit. The following recent result of I.~Emmanouil and P.~Manousaki shows that this is indeed the case.

\begin{theorem}[\cite{EM}, Theorem 2.2]
 There is an exact sequence
 \[
 0 \lra \underset{i}{\varprojlim}^{1} \Tor_{\bullet + i +1}(A, \Sigma^{i} \blank)
 \lra V_{\bullet}(A, \blank) \lra \M_{\bullet}(\Tor(A, \blank)) \lra 0. \qed
 \]
\end{theorem}

 \section{(Co)completeness of finitely presented functors}\label{S:(co)comp} 
Our goal now is to show that the category $\fp$ of finitely presented functors is complete and cocomplete. It appears that these results, stated in the  different language of left-exact sequences, were first established by Ron Gentle~\cite[Remark 1.3. (b)]{Gen}. The proofs presented here are different and slightly more precise: we work directly in the functor categories and show how the (co)limits can be computed.

\subsection{The category of finitely presented covariant functors is complete}

Yoneda's lemma classifies natural transformations from a representable functor to an arbitrary functor. It is also possible to describe natural transformations going in the opposite direction when the arbitrary functor is replaced by a finitely presented one.
Thus, let $(X, \blank) \lra (Y, \blank) \lra F \lra 0$ be exact, where the first map is of the form $(f, \blank)$ for some $f : Y \lra X$. Recall that $\Ker f$ is called the defect 
of~$F$ and is denoted by $w(F)$. Given a representable functor $(Z, \blank)$ we take  
natural transformations into it from the above presentation of $F$, which results in an exact sequence
\[
0 \lra \big(F, (Z, \blank) \big) \lra (Z, Y) \lra (Z, X).
\]
On the other hand, mapping $Z$ into the exact sequence $0 \lra w(F) \lra Y \lra X$, we have an exact sequence 
\[
0 \lra (Z, w(F)) \lra (Z, Y) \lra (Z, X).
\]
Comparing the two sequences, we have
\begin{lemma}
There is a binatural isomorphism 
\[
\big(F, (Z, \blank) \big) \cong \big(Z, w(F)\big).
\]
In other words, the contravariant Yoneda embedding $\mathsf{Y} : \LMod \lra \fp$ and the defect are adjoint to each other on the right.\footnote{It is not difficult to show that this isomorphism is canonical, i.e., independent of the chosen presentation of $F$. This is the reason for using the $\cong$ sign rather than $\simeq$.} \qed
\end{lemma}

For later use, we also recall
\begin{lemma}
 The defect and the zeroth left-derived functor of the contravariant Yoneda's embedding are adjoint to each other on the left, i.e., for any finitely presented covariant functor $F$ and any module $A$ there is a binatural isomorphism 
 \[
 \big(L^{0}\mathsf{Y}(A), F \big) \simeq \big(w(F), A \big).
 \]
\end{lemma}

\begin{proof}
Follows from the definition of the zeroth left-derived functor and the fact that 
 $\big(w(F), \blank \big) \simeq R^{0}F$.
\end{proof}

Combining the previous two lemmas, we have

\begin{proposition}
 The defect is a (contravariant) biadjoint. In particular, it interchanges limits and colimits.\qed
\end{proposition}

 \begin{proposition}\label{P:fp-lim}
 The category of finitely presented covariant functors on $\Lambda$-modules is complete and limits can be computed componentwise. 
 \end{proposition}
 
\begin{proof}
 The category of finitely presented covariant functors is abelian with kernels (and cokernels) defined componentwise. Thus it suffices to show that this category has products and that products can be computed componentwise. This is true for products of representable functors. More precisely, we claim that  the desired product $\prod (X_{i}, \blank)$ is just $(\coprod X_{i}, \blank)$\footnote{The reader is cautioned against making a hasty claim that this is obvious. This is not obvious and requires a proof because the product on the left should be taken in a functor category.} with structure maps induced by the canonical injections $\iota_{i} : X_{i} \to \coprod X_{i}$.
 To see that, let $F$ be a finitely presented functor and suppose we have a family of natural transformations $\ga_{i} : F \lra (X_{i}, \blank)$. As we just saw, each $\ga_{i}$ is uniquely determined by an element of  $(X_{i}, w(F))$, which we denote again by 
 $\ga_{i}$. These elements give rise to a unique  $\beta \in (\coprod X_{i}, w(F))
 \cong \big(F, (\coprod X_{i}, \blank)\big)$ such that $\gb \iota_{i} = \ga_{i}$ for each $i$. 
 Switching back to functors and natural transformations, we have a unique $\gb$ such that $(\iota_{i}, \blank) \gb = \ga_{i}$ for each $i$. This establishes the claim for representable functors.\footnote{Using a more conceptual language, we have just shown that the contravariant Yoneda embedding converts coproducts into products.} 
 Moreover, as the contravariant $\Hom$ converts coproducts in each component to a product in abelian groups, the products of representables can be computed componentwise. Since AB4* holds in the category of abelian groups, it now follows that products of finitely presented functors exist and are computed componentwise.
\end{proof}

\subsection{The category of finitely presented covariant functors is cocomplete}

\begin{lemma}\label{L:fp-com-products}
 Finitely presented covariant functors commute with products.
\end{lemma}

 \begin{proof}
 This follows from the facts that covariant representable functors have this property and products preserve epimorphisms in abelian groups.
\end{proof}

\begin{theorem}\label{T:cocomplete}
 The category $\fp$ of finitely presented functors is cocomplete.
\end{theorem}

\begin{proof}
It is convenient to introduce the following notation. Given a natural transformation 
 $\alpha : (A, -) \to F$, set $\mathbf{b}_{\alpha} := \alpha_{A}(1_{A}) \in F(A)$. By Yoneda's lemma, this is the element that uniquely determines $\alpha$.
 
First we show that $\fp$ has coproducts and we begin with coproducts of representables. Let $\{X_{i}\}_{i \in I}$ be an arbitrary family of modules. Associated with it is the family $\{(X_{i}, \blank)\}_{i \in I}$ of projectives in $\fp$.  Let $\pi_{j} : \prod X_{i} \lra X_{j}$ be the canonical projections. We now claim that the family 
 \[
 (\pi_{j}, -) :  (X_{j}, \blank) \lra (\prod X_{i}, \blank)
 \]
 is a coproduct of $\{(X_{i}, \blank)\}_{i \in I}$. To this end, for any functor $F$ and any family of natural transformations $\beta_{j} : (X_{j}, \blank) \lra F$, we need to find a natural transformation $\alpha : (\prod X_{i}, \blank) \lra F$ making each diagram
 \[
\begin{tikzcd}
 	(X_{j}, \blank) \ar[rr, "{(\pi_{j}, -)}"] \ar[rd, "\beta_{j}"']
	&
	& (\prod X_{i}, \blank) \ar[dl, dashrightarrow, "\alpha"]
\\
	& F
\end{tikzcd}
\]
commute. Each $\beta_{j}$ is determined by $\mathbf{b}_{\beta_{j}} \in F(X_{j})$. By Lemma~\ref{L:fp-com-products}, the canonical map $f : F(\prod X_{i}) \lra \prod F(X_{i})$ 
in the commutative diagram 
\[
\begin{tikzcd}
 	F(\prod X_{i}) \ar[r, "F(\pi_{j})"] \ar[d, "f"' , "\cong"]
	& F(X_{j})
\\
	\prod F (X_{i}) \ar[ur, "p_{j}"']
\end{tikzcd}
\]
is an isomorphism. We can now define $\alpha$ by setting $\mathbf{b}_{\alpha} := f^{-1}(\prod \mathbf{b}_{\beta_{i}})$. By Yoneda's lemma,
it suffices to check that  $ F(\pi_{j})(\mathbf{b}_{\alpha}) = \mathbf{b}_{\beta_{j}}$ for each $j$, i.e., 
\[
F(\pi_{j}) \big(f^{-1}(\prod \mathbf{b}_{\beta_{i}})\big) = \mathbf{b}_{\beta_{j}},
\]
which is immediate from the commutative diagram above. We have thus shown that  the family $ (\pi_{i}, \blank) : (X_{i}, \blank) \lra ( \prod X_{i}, \blank)$ is a coproduct of the
$(X_{i}, \blank)$ and, in particular, the  coproduct is represented by $\prod X_{i}$.\footnote{Using a more conceptual language, we have just shown that the contravariant Yoneda embedding converts products into coproducts.}

Now we can move on to the case of arbitrary  finitely presented functors. Let $\{F_{_{i}}\}_{i \in I}$ be a family of finitely presented functors with presentations
\[
(Y_{i}, \blank) \lra (X_{i}, \blank) \lra F_{i} \lra 0
\]
We claim that the cokernel of the induced natural transformation 
\[
\coprod (Y_{i}, \blank) \lra \coprod (X_{i}, \blank)
\]
is the desired coproduct of the $F_{i}$. This follows from a general fact: if in an abelian category there is a coproduct of a family of morphisms, then the cokernel of this coproduct is the coproduct of the corresponding cokernels. In summary, we have a defining exact sequence
\begin{equation}\label{D:}
 (\prod Y_{i}, \blank) \lra  (\prod X_{i}, \blank) \lra \coprod F_{i} \lra 0
\end{equation}
 As a result, we have that the category of finitely presented functors has coproducts. On the other hand, analogous to equalizers, coequalizers exist in this category. It now follows that $\fp$ is cocomplete.
\end{proof}

\begin{remark}
 Because the direct product cannot be taken out of the contravariant argument of the 
 $\Hom$ functor (i.e., the contravariant $\Hom$ functor does not convert direct products into direct sums), the coproduct, and therefore colimits, of finitely presented functors are not computed componentwise.
\end{remark}

\section{Coherence and defect of the inert asymptotic stabilization}\label{S:coherence}

\subsection{Coherence of the inert asymptotic stabilization}

We now turn attention to the inert univariate functor determined by 
$\T$ 
and give a sufficient condition for it to be finitely presented. We begin by recalling some known (at least to the experts) preliminary results. Recall that a functor is said to be finitely presented if it is a cokernel of a natural transformation between representable functors.

\begin{lemma}\cite[Lemma 6.1]{A66}\label{L:fp-tensor}
 If the left $\Lambda$-module $B$ is finitely presented, then so is the functor 
 $\blank \otimes B$.\footnote{The converse is also true, [ibid.].}
\end{lemma}

\begin{proof}
 Let $P_{1} \lra P_{0} \lra B \lra 0$ be a finite presentation. By the right-exactness of the tensor product, the sequence $\blank \otimes P_{1} \lra \blank \otimes P_{0} \lra  \blank \otimes B \lra 0$ is  exact. The duality for finitely generated projective modules yields a finite presentation $ (P_{1}^{\ast}, \blank ) \lra (P_{0}^{\ast}, \blank ) \lra  \blank \otimes B \lra 0$.
\end{proof}

\begin{lemma}\label{L:tor-1-fp}
 If $B$ is $FP_{2}$, then $\Tor_{1}(\blank, B)$ is finitely presented. More generally, if $B$ is $FP_{n+1}$, $n \geq 1$, then $\Tor_{n}(\blank, B)$ is finitely presented. 
\end{lemma}

\begin{proof}
By assumption, we have a syzygy sequence $0 \lra \Omega B \to P \to B \to 0$, where all modules are finitely presented. The corresponding long exact sequence
\[
0 \lra \Tor_{1}(\blank, B) \lra \blank \otimes \Omega B \lra \blank \otimes P \lra \blank \otimes B \lra 0
\]
and Lemma~\ref{L:fp-tensor} show that $\Tor_{1}(\blank, B)$, being the kernel of a natural transformation between finitely presented functors, is finitely presented. 
The general case can now be treated by dimension shift.
\end{proof}

\begin{proposition}\label{P:inert-fp}
 Suppose that a left $\Lambda$-module $B$ is $FP_{\infty}$ and that it has an injective resolution $(I^{i}, d^{i})$ such that all $I^{i}$ are also $FP_{\infty}$. Then, for all nonnegative integers $l$ and $i$, the functors  $\Omega^{l} \blank \ot \Sigma^{i}B$, where the cosyzygy modules of $B$ are computed relative to that resolution, are finitely presented.
\end{proposition}

\begin{proof}
The cosyzygy sequences $0 \lra \Sigma^{i}B \lra I^{i} \lra \Sigma^{i + 1}B \lra 0$ show that all cosyzygy modules of $B$ are $FP_{\infty}$. Each such sequence gives rise to a long exact sequence of Tor functors, which yields a presentation
 \[
  \Tor_{1}(\Omega^{l} \blank, I^{i}) \lra \Tor_{1}(\Omega^{l} \blank, \Sigma^{i+1}B) \lra 
 \Omega^{l} \blank \ot \Sigma^{i}B \lra 0.
 \]
 Rewriting it as 
 \[
 \Tor_{1}(\blank, \Omega^{l} I^{i}) \lra \Tor_{1}(\blank, \Omega^{l} \Sigma^{i+1}B) \lra \Omega^{l} \blank \ot \Sigma^{i}B \lra 0
 \]
and using Lemma~\ref{L:tor-1-fp}, we have the desired result.
\end{proof}

\begin{theorem}
 Suppose that a left $\Lambda$-module $B$ is $FP_{\infty}$ and that it has an injective resolution all of whose terms are also $FP_{\infty}$. Then the inert asymptotic stabilizations $\T_{n}(\blank, B)$, $n \in \Z$, are finitely presented.
\end{theorem}

\begin{proof}
 By dimension shift, it suffices to assume that $n=0$. Since $\T_{0}$ is a bifunctor, for any right $\Lambda$-module $A$, we have $\T_{0}(\blank, B)(A) \simeq 
 \T_0(A,\blank)(B)$, which is
 $\underset{k \geq 0}{\varprojlim}\, (\Omega^{k} A\ \ot\ \Sigma^k B)$. By Proposition~\ref{P:fp-lim}, the latter is just 
 $\underset{k \geq 0}{\varprojlim}\, (\Omega^{k} \blank\ \ot\ \Sigma^k B)(A)$. The constructed isomorphism is functorial in~$A$ and therefore we have a functor isomorphism 
\[
\T_{0}(\blank, B) \simeq \underset{k \geq 0}{\varprojlim}\, (\Omega^{k} \blank\ \ot\ \Sigma^k B). 
\]
Propositions~\ref{P:inert-fp} and~\ref{P:fp-lim} now show that $\T_{0}(\blank, B)$ is finitely presented.
\end{proof}

As an immediate consequence of the just proved result, we have

\begin{theorem}
 If $B$ is a finitely generated module over an artin algebra, then the functors
 $\T_{n}(\blank, B)$, $n \in \Z$ are finitely presented. \qed
 \end{theorem}
 
 \subsection{The defect of the inert asymptotic stabilization}
 
 In view of the foregoing theorems, it is natural to try and describe the defect of the inert stabilization when it is finitely presented. To this end, we first establish an auxiliary result. Given a right module $A$, choose a syzygy sequence $ 0 \to \Omega A \to P \to A \to 0 $. Given a left module $B$, choose a cosyzygy sequence 
 $ 0 \to B \to I \to \Sigma B \to 0 $ and a syzygy sequence
 $0 \to \Omega \Sigma B \to Q \to \Sigma B \to 0 $. Lifting the identity map on 
 $\Sigma B$, we have a commutative diagram 
 
\begin{equation}\label{Eq:double-decker}
\begin{tikzcd}
 	0 \ar[r]
	& \Omega \Sigma B \ar[r] \ar[d]
	& Q \ar[r] \ar[d]
	& \Sigma B \ar[r] \ar[d, equals]
	& 0
\\
	0 \ar[r]
	& B \ar[r]
	& I \ar[r]
	& \Sigma B \ar[r] 
	& 0
\end{tikzcd}
\end{equation}
In the leftmost vertical map we replace $B$ with $\Sigma B$ to obtain a map 
$\Omega \Sigma^{2} B \to \Sigma B$. Applying $\Omega$, we have a map 
$\Omega^{2} \Sigma^{2} B \to \Omega \Sigma B$. Iterating this process and applying the functor $\Tor_{1}(A, \blank)$ we have a sequence
\[
\dots \lra \Tor_{1}(A, \Omega^{2} \Sigma^{2} B) \lra \Tor_{1}(A, \Omega \Sigma B) \lra 
 \Tor_{1}(A, B).
\]
 Finally, replacing $B$ with $\Sigma B$, we have the sequence
\begin{equation}\label{Eq:o-s}
\dots \lra \Tor_{1}(A, \Omega^{2} \Sigma^{3} B) \lra \Tor_{1}(A, \Omega \Sigma^{2} B) \lra  \Tor_{1}(A, \Sigma B).
 \end{equation}
 Now we recall the second construction of the asymptotic stabilization and the intertwining diagram~\eqref{Eq:intertwine}.

\begin{lemma}\label{L:tor-1}
 The sequence~\eqref{Eq:o-s} is isomorphic to the top row of~\eqref{Eq:intertwine}.
\end{lemma}

\begin{proof}
 The proof can be accomplished by tensoring the diagram~\eqref{Eq:double-decker} with the syzygy sequence $ 0 \to \Omega A \to P \to A \to 0 $, doing a diagram chase, and using the balance of the bifunctor $\Tor$. The tedious but more or less straightforward details are left to the reader.
\end{proof}

Before we can describe the defect of the asymptotic stabilization, we need to compute the defect of the univariate $\Tor$ functor.

\begin{lemma}\label{L:w-tor}
 Suppose $B \in \modL$ is $FP_{2}$. Then $w(\Tor_{1}(\blank, B)) \simeq
 \Ext^{1}(B, \Lambda)$.
\end{lemma}

\begin{proof}
 Since $B$ is $FP_{2}$, we can find a syzygy sequence $0 \to \Omega B \to Q \to B \to 0$ all of whose terms are finitely presented. Dualizing it into $\Lambda$ we have an exact sequence 
\[
0 \to B^{\ast} \to Q^{\ast} \to (\Omega B)^{\ast} \to \Ext^{1}(B, \Lambda) \to 0.
\]
On the other hand, the same syzygy sequence gives rise to an exact sequence of functors
 \[
 0 \lra \Tor_{1}(\blank, \Omega B) \lra \blank \otimes \Omega B \lra \blank \otimes Q 
 \lra \blank \otimes B \lra 0.
 \]
 Since $w$ is a biadjoint~\cite[The diagram after Theorem 4.2]{MR-2}, it is exact and, applying it to the sequence above, we have an exact sequence 
 \[
 0 \to B^{\ast} \to Q^{\ast} \to (\Omega B)^{\ast} \to w \big(\Tor_{1}(\blank, \Omega B)\big) \to 0.
 \]
 The result now follows.
\end{proof}

We are now ready to describe the defect of the asymptotic stabilization.

\begin{theorem}
Suppose that a left $\Lambda$-module $B$ is $FP_{\infty}$ and that it has an injective resolution all of whose terms are also $FP_{\infty}$. Then 
\[
w \big(\T_{0}(\blank, B)\big) \simeq \varinjlim  \Ext^{1} (\Omega^{i}\Sigma^{i+1}B, \Lambda)
\] 
where the limit is taken over the sequence of iterations of the map 
from~\eqref{Eq:double-decker}. In particular, this formula applies to an arbitrary finitely generated module over an arbitrary artin algebra.
\end{theorem}

\begin{proof}
  
\begin{align*}
 w \big(\T_{0}(\blank, B)\big)
 	& \simeq w\big(\varprojlim \Tor_{1} (\Omega^{i}\blank, \Sigma^{i+1}B)\big)
	& \text{(by definition)}
\\
	& \simeq w\big(\varprojlim \Tor_{1} (\blank, \Omega^{i}\Sigma^{i+1}B)\big)
	& \text{(by Lemma~\ref{L:tor-1})}
\\
	& \simeq \varinjlim w\big( \Tor_{1} (\blank, \Omega^{i}\Sigma^{i+1}B)\big) 
	& \text{(since $w$ converts limits to colimits)}
\\
	& \simeq \varinjlim  \Ext^{1} (\Omega^{i}\Sigma^{i+1}B, \Lambda)
	& \text{(by Lemma~\ref{L:w-tor})}
\end{align*}
\end{proof}

\begin{remark}\label{R:W}
 The defect in question can also be written as 
 \[
 w \big(\T_{0}(\blank, B)\big) \simeq \varinjlim  \Ext^{1} (\Sigma^{i+1}B, \Sigma^{i}\Lambda).
 \] 
 Replacing the syzygy endofunctor on the projectively stable category by the cosyzygy endofunctor on the injectively stable category, and universally inverting $\Sigma$ we have what we may call Buchweitz cohomology $W^{\bullet}$ based on injectives. (See~\cite{N-98} for more details on this construct.) Arguments similar to the ones preceding Theorem~\ref{T:second} show that the universal inversion of $\Sigma$ can be replaced by a stabilization of $\Ext^{1}$ (which replaces $\Hom$ modulo injectives). This leads to a surprising description of the defect:
 \[
 w \big(\T_{0}(\blank, B)\big) \simeq W^{0}(B, \Lambda) = \varinjlim\, (\overline{\Sigma^{i}B, \Sigma^{i}\Lambda}).
 \]
 The reader is invited to compare this formula with that for the defect of the tensor product~(\cite[Example 3.14]{MR-1}): $w(\blank \otimes B) \simeq (B, \Lambda) = B^{\ast}$.
\end{remark}

The just proved theorem immediately leads to a description of the right-derived functors of $\big(\T_{0}(\blank, B)\big)$. To see this, recall (\cite[top of page 210]{A66}) that, for a finitely presented functor $F$, the natural transformation $R^{0}F \to (w(F), \blank)$ is an isomorphism. Since, for each $n$, the natural transformation $R^{n}F \lra R^{n}R^{0}F$ is always an isomorphism, we have
\begin{corollary}\label{C:Rn}
 Under the assumptions of the theorem, the natural transformation 
 \[
 R^{n}\big(\T_{0}(\blank, B)\big) \lra \Ext^{n}\big(\varinjlim  \Ext^{1} \big(\Omega^{i}\Sigma^{i+1}B, \Lambda), \blank\big) \simeq \Ext^{n}\big(W^{0}(B, \Lambda), \blank \big)
 \]
 is an isomorphism for all $n$. \qed
\end{corollary}

In view of Remark~\ref{R:W}, we have

\begin{corollary}
 Under the assumption of the theorem, if $\Lambda$ is of finite injective dimension as a left module over itself, then all derived functors of $\T_{0}(\blank, B)$ are zero. \qed
\end{corollary}

Now recall the torsion functor $\injt =  \blank \ot \Lambda$ introduced in~\cite{MR-2}. The foregoing discussion motivates
\begin{definition}
 The  functor 
\[
\injt_{\infty} := \T_{0}(\blank, \Lambda) \simeq \underset{k \geq 0}{\varprojlim}\, (\Omega^{k} \blank\ \ot\ \Sigma^k \Lambda)
 \] 
is called the \texttt{asymptotic torsion} functor.\footnote{This is a functor on right modules. A similar definition applies to left modules.}
\end{definition}

For the next result, recall the cotorsion functor 
$\injc = (\overline{\Lambda, \blank})$ introduced in~\cite{MR-2} and defined on left modules. It is now natural to introduce 
\begin{definition}
 The functor 
\[
\injc^{\infty} := W^{0}(\Lambda, \blank) = \varinjlim\, (\overline{\Sigma^{i}\Lambda, \Sigma^{i}\blank})
\]
may therefore be called the \texttt{asymptotic cotorsion} functor.\footnote{A similar definition applies to right modules.} 
\end{definition}

\begin{theorem}
 Suppose that $\Lambda$, viewed as a left module over itself, has an injective resolution all of whose terms are $FP_{\infty}$. Then 
 \[
 w(\injt_{\infty}) \simeq \injc^{\infty}(\prescript{}{\Lambda}\Lambda),
 \]
 i.e., the defect of the asymptotic torsion on right modules is isomorphic to the asymptotic cotorsion of $\Lambda$ viewed as a left module over itself. \qed
\end{theorem}

\begin{remark}
 The reader should compare this formula with~\cite[Corollary~5.4]{MR-2} showing that 
 \[
 w(\injt) \simeq \injc(\prescript{}{\Lambda}\Lambda)
 \]
when the injective envelope of $\prescript{}{\Lambda}\Lambda$ is finitely presented.
\end{remark}

Specializing Corollary~\ref{C:Rn} to $B:= \prescript{}{\Lambda}\Lambda$, we have

\begin{corollary}
 Under the assumptions of the theorem, the canonical natural transformation 
 \[
 R^{n}\injt_{\infty} \lra \Ext^{n}_{\Lambda}\big(\injc^{\infty}(\prescript{}{\Lambda}\Lambda), \blank\big)
 \]
 is an isomorphism for all $n$.
\end{corollary}

\end{document}